%% file: main.tex
\pdfoutput=1
\documentclass[a4paper,12pt]{article}

\input{header}

\begin{document}
\maketitle
\vspace{-1cm}
\input{abstract}

\setcounter{tocdepth}{2}
\tableofcontents
\newpage

\input{sec_intro}
\input{sec_prelim}
\input{sec_fatDelta}
\input{sec_monad}

\printbibliography
\end{document}

%% file: header.tex
\usepackage[english]{babel}
\usepackage[dvipsnames]{xcolor} 
\usepackage{amsmath,amssymb,amsthm,mathtools,newtxmath,extpfeil} 
\usepackage{mdframed} 
\usepackage{extarrows}
\usepackage{multirow} 
\usepackage{microtype} 
\usepackage{xifthen} 
\usepackage{pict2e} 
\usepackage{enumerate} 
\usepackage{enumitem} 
\usepackage{calligra} 
\usepackage{csquotes}
\usepackage[font={footnotesize},labelfont={footnotesize}]{caption} 

\usepackage{authblk} 

\renewcommand\Affilfont{\normalsize}
\makeatletter
\renewcommand\AB@affilsepx{, \protect\Affilfont}
\makeatother

\usepackage[
  colorlinks=true,  
  linktocpage=true, 
  linkcolor=NavyBlue,    
  citecolor=Green,  
  urlcolor=BrickRed, 
]{hyperref} 
\usepackage[capitalize]{cleveref} 

\usepackage[
  style=alphabetic,
  maxbibnames=5,
  backend=biber
]{biblatex} 
\DeclareFieldFormat{postnote}{#1}
\DeclareFieldFormat{multipostnote}{#1}
\usepackage{geometry} 
 \usepackage{fancyhdr}
\usepackage{tocloft}
\cftsetindents{section}{0em}{1.5em}
\cftsetindents{subsection}{0em}{2em}
\usepackage{titlesec} 
\titleformat*{\section}{\large\scshape\center}
\titleformat*{\subsection}{\normalsize\scshape}
\titleformat*{\subsubsection}{\normalsize}

\usepackage{tikz-cd}
\usetikzlibrary{
  decorations.markings,
  decorations.pathreplacing,
  shapes.geometric,
  matrix,arrows,
  chains,positioning,
  scopes
}

\pgfdeclarearrow{
  name = pxto,
  setup code = {
    \pgfarrowssettipend{1.5\pgflinewidth}
    \pgfarrowssetbackend{-2.5508\pgflinewidth}
    \pgfarrowssetlineend{-.25\pgflinewidth}
    \pgfarrowssetvisualbackend{-0.021\pgflinewidth}
    \pgfarrowsupperhullpoint{1.5\pgflinewidth}{0\pgflinewidth}
    \pgfarrowsupperhullpoint{-2.0085\pgflinewidth}{3.6525\pgflinewidth}
    \pgfarrowsupperhullpoint{-2.5508\pgflinewidth}{3.0763\pgflinewidth}
  },
  drawing code = {
    \pgfsetdash{}{0pt}%
    \pgfpathmoveto{\pgfpoint{1.5\pgflinewidth}{0.0254\pgflinewidth}}%
    \pgfpathlineto{\pgfpoint{-2.0085\pgflinewidth}{3.6525\pgflinewidth}}%
    \pgfpathlineto{\pgfpoint{-2.5508\pgflinewidth}{3.0763\pgflinewidth}}%
    \pgfpathlineto{\pgfpoint{-0.4322\pgflinewidth}{0.5\pgflinewidth}}%
    \pgfpathlineto{\pgfpoint{-0.4322\pgflinewidth}{-0.5\pgflinewidth}}%
    \pgfpathlineto{\pgfpoint{-2.5508\pgflinewidth}{-3.0763\pgflinewidth}}%
    \pgfpathlineto{\pgfpoint{-2.0085\pgflinewidth}{-3.6525\pgflinewidth}}%
    \pgfpathclose%
    \pgfusepathqfill
  }
}

\tikzset{>=pxto}
\tikzcdset{arrow style=tikz}
\tikzset{mid vert/.style={
  /utils/exec=\tikzset{every node/.append style={outer sep=0.8ex}},
  postaction=decorate,
  decoration={
    markings,
    mark=at position 0.5 with {\draw[solid,-] (0,#1) -- (0,-#1);}}},
mid vert/.default=0.75ex}

\tikzset{
  emptyNode/.style={shape=coordinate} 
}

\DeclareSymbolFont{yhlargesymbols}{OMX}{yhex}{m}{n}
\DeclareMathAccent{\yhwidehat}{\mathord}{yhlargesymbols}{"62}

\DeclareMathOperator{\Cat}{\mathsf{Cat}}

\DeclareMathOperator{\SemiCat}{\mathsf{se}\mathsf{Cat}}
\DeclareMathOperator{\RelSemiCat}{\SemiCat^{\mathsf{\mskip-1mu r}}}
\DeclareMathOperator{\Graph}{\mathsf{Gph}}
\DeclareMathOperator{\RelGraph}{\Graph^{\mathsf{\mskip-1mu r}}}

\DeclareMathOperator{\Iso}{\mathsf{Iso}}

\DeclareMathOperator{\Path}{\mathsf{Path}}
\DeclareMathOperator{\rPath}{\mathsf{Path}^{\mathsf{r}}}
\DeclareMathOperator{\Mono}{\mathsf{Mono}}
\DeclareMathOperator{\op}{\mathsf{op}}
\DeclareMathOperator{\id}{\mathsf{id}}
\DeclareMathOperator*{\colim}{\mathsf{colim}}

\newcommand{\UU}{\mathcal{U}}

\newcommand{\fDel}{\underline{\Delta}}
\newcommand{\Fr}{\mathscr{F}}
\newcommand{\sFr}{\mathsf{F}}
\newcommand{\sForget}{\mathsf{U}}
\newcommand{\rFr}{\mathsf{F}^{\mathsf{r}}}
\newcommand{\rForget}{\mathsf{U}^{\mathsf{r}}}
\newcommand{\free}{\mathtt{f}}
\newcommand{\rfree}{\free^{\mathsf{r}}}
\newcommand{\frc}{\mathsf{fr}}

\newcommand{\Alt}{\mathscr{A}}
\newcommand{\Nerve}{\mathcal{N}}
\newcommand{\fNerve}{\underline{\Nerve}}
\newcommand{\pathp}{\mathsf{p}}
\newcommand{\pathq}{\mathsf{q}}
\newcommand{\Mbo}{\mathsf{bo}}
\newcommand{\oo}{\ensuremath{\infty}\nobreakdash}
\newcommand{\oox}[1]{(\oo,\! $#1$)\nobreakdash}
\newcommand{\A}{\mathcal{A}}

\newcommand{\C}{\mathcal{C}}
\newcommand{\D}{\mathcal{D}}
\newcommand{\E}{\mathcal{E}}
\newcommand{\F}{\mathcal{F}}

\newcommand{\I}{\mathcal{I}}

\newcommand{\M}{\mathcal{M}}

\newcommand{\fVert}{\mathcal{V}}

\newcommand{\src}{\mathsf{s}}
\newcommand{\trg}{\mathsf{t}}

\newcommand{\mlgar}{{\color{red}{\longrightarrow}}}

\newcommand{\vpro}{\vee}
\newcommand{\Psh}[1]{\widehat{#1}}
\newcommand{\fat}[1]{\underline{#1}}
\newcommand{\emb}[1]{\mathscr{#1}}
\newcommand{\forgetful}[1]{#1^{\mathsf{u}}}

\newcommand{\ftop}[1]{\overline{#1}}
\newcommand{\fbot}[1]{\underline{#1}}
\newcommand{\tOb}[1]{\mathtt{1}}
\newcommand{\blank}{\mathord{\hspace{0.25ex}\underline{\hspace{1ex}}\hspace{0.25ex}}}
\newcommand{\marking}[1]{{#1}^{\natural}}
\newcommand{\carrier}[1]{{#1}^{\sqcup}}

\newcommand{\altg}[1]{#1^{1}}
\newcommand{\altl}[1]{#1^{0}}
\newcommand{\alto}[1]{\mathscr{#1}}
\newcommand{\Arr}[1]{#1^{\to}}
\newcommand{\Ar}[2]{\mathsf{Ar}_{#1}(#2)}
\newcommand{\EM}[1]{\mathsf{EM}(#1)}
\newcommand{\Kl}[1]{\mathsf{Kl}(#1)}
\newcommand{\bfm}[1]{\mathbf{#1}}
\newcommand{\genm}[1]{\hat{#1}}
\newcommand{\freem}[1]{\tilde{#1}}

\renewcommand{\:}{\colon}

\makeatletter
\newcommand{\vbto}{}
\newcommand{\vblongto}{}
\newcommand{\xvbto}[1]{\overset{#1}{\vbto}}
\newcommand{\xvblongto}[1]{\overset{#1}{\vblongto}}

\DeclareRobustCommand{\vbto}{\mathrel{\mathpalette\p@to@gets\to}}
\DeclareRobustCommand{\vblongto}{\mathrel{\mathpalette\p@longto@gets\longrightarrow}}

\newcommand{\p@to@gets}[2]{%
  \ooalign{\hidewidth$\m@th#1\mapstochar\mkern5mu$\hidewidth\cr$\m@th#1\to$\cr}%
}
\newcommand{\p@longto@gets}[2]{%
  \ooalign{\hidewidth$\m@th#1\mapstochar\mkern5mu$\hidewidth\cr$\m@th#1\longrightarrow$\cr}%
}
\makeatother

\makeatletter
\newcommand{\adjunction}{\@ifstar\named@adjunction\normal@adjunction}
\newcommand{\normal@adjunction}[4]{%
  #1\colon #2%
  \mathrel{\vcenter{%
    \offinterlineskip\m@th
    \ialign{%
      \hfil$##$\hfil\cr
      \longrightharpoonup\cr
      \noalign{\kern-.3ex}
      \smallbot\cr
      \noalign{\kern.2ex}
      \longleftharpoondown\cr
    }%
  }}%
  #3 \noloc #4%
}
\newcommand{\named@adjunction}[4]{%
  #2%
  \mathrel{\vcenter{%
    \offinterlineskip\m@th
    \ialign{%
      \hfil$##$\hfil\cr
      \scriptstyle#1\cr
      \noalign{\kern.1ex}
      \longrightharpoonup\cr
      \noalign{\kern-.3ex}
      \smallbot\cr
      \longleftharpoondown\cr
      \scriptstyle#4\cr
    }%
  }}%
  #3%
}
\newcommand{\longrightharpoonup}{\relbar\joinrel\rightharpoonup}
\newcommand{\longleftharpoondown}{\leftharpoondown\joinrel\relbar}
\newcommand\noloc{%
  \nobreak
  \mspace{6mu plus 1mu}
  {:}
  \nonscript\mkern-\thinmuskip
  \mathpunct{}
  \mspace{2mu}
}
\newcommand{\smallbot}{%
  \begingroup\setlength\unitlength{.25em}%
  \begin{picture}(1,1)
  \roundcap
  \polyline(0,0)(1,0)
  \polyline(0.5,0)(0.5,1)
  \end{picture}%
  \endgroup
}
\makeatother

\newtheoremstyle{TheoremNum}
{\topsep}{\topsep}              
{\itshape}                      
{}                              
{\bfseries}                     
{.}                             
{ }                             
{\thmname{#1}\thmnote{ \bfseries #3}}

\theoremstyle{TheoremNum}
\newtheorem{thmn}{Theorem}

\theoremstyle{plain}
\newtheorem{thm}{Theorem}[section]
  \Crefname{thm}{Theorem}{Theorems}
\newtheorem{prop}[thm]{Proposition}
  \Crefname{prop}{Proposition}{Propositions}
\newtheorem{cor}[thm]{Corollary}
  \Crefname{cor}{Corollary}{Corollaries}
\newtheorem{lem}[thm]{Lemma}
  \Crefname{lem}{Lemma}{Lemmas}

\theoremstyle{definition}

  \Crefname{dfn}{Definition}{Definitions}
\newtheorem{constr}[thm]{Construction}
  \Crefname{constr}{Construction}{Constructions}

\theoremstyle{remark}
\newtheorem{rmk}[thm]{Remark}
  \Crefname{rmk}{Remark}{Remarks}

\title{
		A study of Kock's fat Delta
	   }
\author[1]{Tom de Jong}
\author[1]{Nicolai Kraus}
\author[2]{Simona Paoli}
\author[1]{Sti\'{e}phen Pradal}
\affil[1]{University of Nottingham}
\affil[2]{University of Aberdeen}
\date{\vspace{-1em}
		\normalsize March 14, 2025
	 }

%% file: abstract.tex
\abstract{
Motivated by the study of weak identity structures in higher category theory we explore the fat Delta category, a modification of the simplex category introduced by J.\ Kock and provide a comprehensive study via the theory of monads with arities.
Specifically, by proving that the free relative semicategory monad is strongly cartesian and identifying a dense generator, the theory of monads with arities immediately gives rise to the nerve theorem.
We characterise the essential image of the nerve via the Segal condition, and show that fat Delta possesses an active-inert factorisation system.
Building on these results, we also establish an isomorphism between two presentations of fat Delta.
}


%% file: sec_intro.tex
\addtocontents{toc}{\protect\setcounter{tocdepth}{1}}
\section{Introduction}
Higher categories come with many levels of structure, and each of the various approaches to higher category theory needs to provide a way to organise, represent and work with this structure.
One very fruitful approach to higher categories has been to use simplicial and multi-simplicial objects, that is functors from the opposite of the simplicial category $\Delta$ and its products, satisfying additional conditions that encode the composition of higher morphisms and their behaviour.

This simplicial approach has given rise to models of \oox{1}-categories (in particular \emph{quasi-categories}~\cite{Joyal2008,LurieHTT} and \emph{complete Segal spaces}~\cite{Rezk2001}) and, more generally, models of \oox{n}-categories~\cite{Bergner2011}, that is categories with morphisms in all dimensions but weakly invertible in dimension higher than $n$.
These have found important applications to mathematical physics in the field of TQFTs~\cite{Lurie2008,Bergner2011,GP2022}.
Multi-simplicial models have also been used to describe truncated higher categories, that is those with morphisms only in dimension up to $n$.
These arise naturally in homotopy theory, where, in the higher groupoidal case, they describe the building blocks of spaces, called homotopy $n$-types.
Such models include the Tamsamani model \cite{Tamsamani1999} and the weakly globular $n$-fold categories of one of the authors \cite{Paoli2019}.

The current paper studies the so-called \emph{fat Delta} category ($\fDel$, cf.~\cref{fig:del-vs-fdel}), a modification of $\Delta$ introduced by J.\ Kock~\cite{Kock2006} with the motivation of weakening the identity structure of higher categories.
In fact, $\fDel$ can be used as a means to give a diagrammatic interpretation of weak identity arrows in higher categories.
A variation of $\fDel$ was later (independently) suggested by C.\ Sattler and one of the current authors~\cite{KS2017}, with a very different motivation stemming from type theory.
Before describing the focus of this paper, we explain the motivation for our interest in this category, which comes from different areas.

\newcommand\x{1mm}
\newcommand\xx{2mm}

\newcommand\y{1.5mm}
\newcommand\yy{3mm}

\begin{figure}[h]
\begin{mdframed}
\centering
	\begin{minipage}{.5\textwidth}
	\centering
	\begin{tikzcd}[row sep=1.25cm, column sep=1.05cm, nodes={inner sep=2pt}]
		{[0]}
			\ar[r, shift left=\y]
			\ar[r, shift left=-\y]
		&
		{[1]}
			\ar[l]
			\ar[r, shift left=\yy]
			\ar[r, shift left=0mm]
			\ar[r, shift left=-\yy]
		&
		{[2]}
			\ar[l, shift right=\y]
			\ar[l, shift right=-\y]
	\end{tikzcd}
	\end{minipage}%
	\begin{minipage}{.5\textwidth}
	\begin{tikzcd}[row sep=0.85cm, column sep=0.5cm]
		{[0]}
			\ar[rr, shift left=\x]
			\ar[rr, shift left=-\x]
			\ar[d, shift left=\x]
			\ar[d, shift left=-\x]
		&&
		{[1]^{\flat}}
			\ar[rr, shift left=\xx]
			\ar[rr, shift left=0mm]
			\ar[rr, shift left=-\xx]
			\ar[dll]
			\ar[d, shift left=\x]
			\ar[d, shift left=-\x]
			\ar[dr, shift left=\x]
			\ar[dr, shift left=-\x]
		&&
		{[2]^{\flat}}
			\ar[dll]
			\ar[dl, bend left=10]
		\\
		{[1]^{\sharp}}
			\ar[d, shift left=\xx]
			\ar[d, shift left=0mm]
			\ar[d, shift left=-\xx]
			\ar[rr, shift left=\x]
		&&
		{\sigma^{1}_{0}}
			\ar[dll]
		&
		{\sigma^{1}_{1}}
			\ar[dlll, bend left=10]
			\ar[from=lll, shift left=0, bend right=20, crossing over]
			\ar[from=ul, shift left=\x, crossing over]
			\ar[from=ul, shift left=-\x, crossing over]
		\\
		{[2]^{\sharp}}
	\end{tikzcd}
	\end{minipage}
	\caption{ Initial segments of the categories $\Delta$ (left) and its modification $\fDel$ (right) by Kock~\cite{Kock2006}.
	While $\Delta$ is a Reedy category (having positive and negative morphisms which go ``up'' and ``down''), $\fDel$ is a direct category (its morphisms only go ``up'').
	The objects of $\fDel$ are the epimorphisms of $\Delta$: the top row of $\fDel$ are the identities, the second row shows the epimorphisms of $\Delta$ that are drawn in the diagram on the left, and the third row represents $[2] \twoheadrightarrow [0]$.
	}\label{fig:del-vs-fdel}
\end{mdframed}
\end{figure}
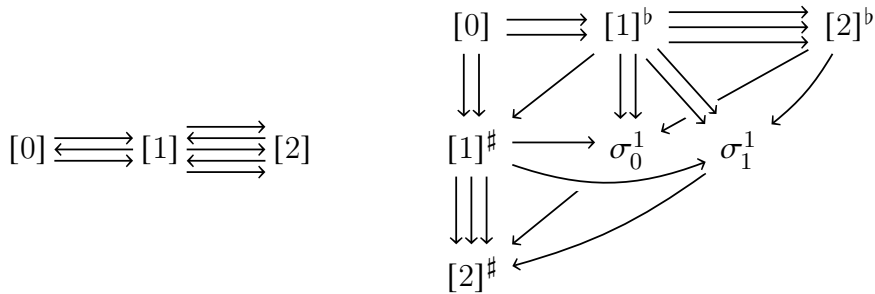

\subsection{Motivation from higher category theory}
It is well known that strict $n$-categories (that is, those where higher morphisms have compositions that obey associative and unit laws strictly) are insufficient for many applications.
For instance, strict $n$-groupoids cannot model homotopy $n$-types. Explicit examples have been found in which strict $3$-groupoids cannot model the homotopy $3$-types of spaces with non-trivial Whitehead products~\cite{Sim1998}.
Likewise, higher order cobordism categories necessitate weak higher categories~\cite[\S2.2]{Lurie2008}: that is, those where compositions of higher morphisms obey associative and unit laws only up to an invertible higher morphism in the next dimension, and these associativity and unit isomorphisms are sufficiently compatible or coherent.

C.\ Simpson conjectured in~\cite{Sim1998} that any such model of weak $n$-category should be suitably equivalent to one in which the composition laws are strictly associative, that is they hold as equalities, and the unit laws hold up to coherent isomorphism.
This is usually referred to as the \emph{weak units conjecture} for truncated weak $n$-categories.
J.\ Kock used the fat Delta category to give a precise formulation of the weak units conjecture (namely, as stating that the Tamsamani $n$-categories are suitably equivalent to fair $n$-categories)~\cite{Kock2006}.
The category of fair $n$-categories is built from functors from the opposite of $\fDel$ and its products in an inductive way, formally analogous to the Tamsamani model, but with the simplex category replaced by fat Delta.
A main difference however is that in the fair $n$-categories model, all Segal maps are isomorphisms, hence all compositions of higher morphisms are (strictly) associative, while the unit laws hold only up to isomorphisms.
The combinatorics of fat Delta allows precisely to decouple the units structure from the composition structure to make this possible, unlike in the simplicial approach.

For general $n$, this conjecture is open.
However, in the case $n=2$ this has been proved by one of the authors using the intermediary of weakly globular double categories \cite{Pao2025}, and it is hoped that this approach could in the future be extended to higher dimensions.
For $n = 3$, A.\ Joyal and J.\ Kock have given a proof for one-object 3-groupoids \cite{JK2007}.
More recently, S.\ Henry has given a proof of the conjecture for regular \oo-groupoids~\cite{Henry2018}.
The study of models of weak and semistrict (\oo,1)-categories based on $\fDel$ will appear in the forthcoming PhD thesis of the last author.

We have so far described weak units in the context of a strict composition structure.
It is possible to weaken the unital structure even if the composition structure is already weak.
In the context of infinity categories, studies in this direction include \cite{Harpaz_2015} and \cite{Haugseng_2021}.
This direction is closely related to what we discuss next, namely the motivation for $\fDel$ stemming from type theory.

\subsection{Motivation from type theory}\label{sec:intrott}
A further motivation for studying $\fDel$ comes from \emph{type theory}, a foundation for mathematics that is used as a basis for programming languages and proof assistants~\cite{MartinLof1984,NPS90}.
As a foundational system, type theory is often compared to set theory, where the primitive objects are \emph{sets}.
In contrast, \emph{types} (the primitive objects of type theory) behave more akin to \emph{spaces}, i.e.\ \oo-groupoids or homotopy types, a view that \emph{homotopy type theory (HoTT)}~\cite{hottbook} is built upon.

There have been extensive discussions about the philosophical and practical differences between set theory and type theory (cf.~\cite{Reflections}) that are unimportant for the current paper.
It suffices to understand that, while the predominant approach is to \emph{define} the notion of a space in terms of sets (for example, as a Kan simplicial set), every object in homotopy type theory \emph{is} a space.
In a nutshell, homotopy type theory is a language in which spaces are the primitive objects and all operations one can talk about are invariant under homotopy \cite{Shulman2021}.

In homotopy type theory, the collection of all (small) spaces/types is a so-called universe $\UU$.
When studying \oox{1}-categories, it seems natural to work with complete Segal spaces, with spaces being the objects in $\UU$.
The definition one might expect would be a functor from $\Delta$ to $\UU$, of which one requires Reedy fibrancy as well as a Segal and a completeness condition as additional properties.
The issue is that this formulation requires imposing functor laws and, as strict equalities are not automatically homotopy invariant, the language of homotopy type theory cannot even state them.
\footnote{While it is easily justifiable to internalise strict equality in the language, as done in \emph{two-level type theory}~\cite{ACK2016,ACKS2023}, doing so necessitates manually checking that one stays within the fibrant framework, something that does not seem possible here.}

But it turns out that (contravariant) Reedy fibrant diagrams over certain categories can be defined in a way that never leaves type theory's world of homotopy-invariant constructions~\cite{Shulman2015,ACKS2023}.
For example, a Reedy fibrant space-valued contravariant diagram over the inverse category $\mathcal C$ given by $V \rightrightarrows E$ is fully described by a space $X$ of vertices and, for all pairs $(x_1, x_2)$ in $X \times X$, a space $Y(x_1,x_2)$ of edges.
This is a complete definition of the space of Reedy fibrant functors from $\mathcal C$ to $\UU$ without explicitly mentioning equalities; all components of the usual definition can easily be derived.

Clearly, the strategy employed here can only work if the index category is \emph{direct}, i.e.\ if every sequence $A \leftarrow B \leftarrow C \leftarrow \ldots$ of non-identity morphisms is finite.
The simplex category $\Delta$ is not direct, but $\fDel$ (as well as the similar category suggested by C.\ Sattler and a current author~\cite{KS2017}) is, making it suitable for homotopy type theory.

\subsection{Contributions}
In this paper, we study the category $\fDel$ (\cref{sec:fDeldef}) via the theory of monads with arities (\cref{sec:prelim,sec:frelsem}).
The theory of monads with arities, originating in the work of M.\ Weber and others \cite{Weber2004,Weber2007,BMW2012}, allows us to directly derive a number of results about $\fDel$ that are important for the study of higher structures.
\footnote{It was successfully applied \cite{Weber2007,BMW2012} to a number of categories that proved fundamental for the development of higher structures including $\Delta$, Joyal's disk category $\Theta_n$ and the dendroidal category $\Omega$.}
In particular, in \cref{thm:arfDel}, we identify the category of arities of the free relative semicategory monad, using a non-obvious choice of dense generator.
This allows us to give a direct isomorphism between the two presentations of $\fDel$ described in \cref{sec:fDeldef}:

\begin{thmn}[\ref{thm:fDeldesc}]
	The category of finite non-empty relative semiordinals is isomorphic to $\fDel$.
\end{thmn}

This is an important result as it makes precise the intuition conveyed in \cite{Kock2006,Pao2025}.
We then obtain the nerve theorem for $\fDel$ and characterise the essential image of the nerve functor via the Segal condition.

\begin{thmn}[\ref{thm:fnerve}]
	The nerve functor $\fNerve\:\RelSemiCat \to \Psh{\fDel}$, where $\RelSemiCat$ is the category of relative semicategories, is fully faithful and the essential image is spanned by the presheaves satisfying the Segal condition.
\end{thmn}

As a consequence of the theory of monads with arities, \cref{thm:fDelhomo,cor:actinfDel} imply that $\fDel$ possesses an active-inert factorisation system, a fact for which we also give a direct proof allowing us to give convenient characterisations of the active and inert morphisms.



\subsection{Outline}
In \cref{sec:prelim,sec:fDeldef}, we review the concepts and tools needed in the rest of the paper.
More precisely, we provide an overview of relative semicategories and the theory of monads with arities in \cref{sec:prelim}, along with several (new) technical lemmas concerning the interaction between local right adjoints and dense generators.
\cref{sec:fDeldef} is devoted to a detailed recollection of the definitions of the category $\fDel$.
In \cref{sec:frelsem}, we study $\fDel$ using the theory of monads with arities.
For that, we define both the free semicategory monad on the category of (directed) graphs and the free relative semicategory monad on the category of relative graphs arising from the respective free-forgetful adjunctions.
We show in \cref{prop:rfreeStrgCart} that they are strongly cartesian: this is a key technical requirement to apply the theory.

\subsection*{Acknowledgements}
Tom de Jong, Nicolai Kraus and Sti\'ephen Pradal have been supported by grants from the Royal Society (URF\textbackslash{}R1\textbackslash{}191055, RF\textbackslash{}ERE\textbackslash{}210032, RF\textbackslash{}ERE\textbackslash{}231052, URF\textbackslash{}R\textbackslash{}241007).
Simona Paoli would like to thank the University of Nottingham for its hospitality during a visit in April 2024.
She also thanks Clemens Berger for helpful discussions.
Finally, Sti\'ephen Pradal would like to thank Paul-Andr\'e Melli\`es and also Clemens Berger for insightful discussions on monads with arities.

\addtocontents{toc}{\protect\setcounter{tocdepth}{2}}


%% file: sec_prelim.tex
\section{Preliminaries and General Results}\label{sec:prelim}
In this section, we start by introducing the concept of relative semicategories, which are central to our study.
The category of interest, fat Delta $\fDel$, is deeply entangled with this notion. We will begin by defining semicategories and then proceed to relative semicategories by equipping semicategories with a marking.
In the rest of the section, we recall the background on monads with arities from \cite{Weber2007,BMW2012} and prove new results (\cref{sec:lra}) showing that genericness behaves well with colimits.
These results will be used to study the free relative semicategory monad in \cref{sec:frelsem}.

\subsection{Relative Semicategories}\label{sec:rsc}
A \emph{semicategory} $\C$ (also known as \emph{non-unital category}) consists of a class of \emph{objects} $\C_0$ and of a set of \emph{morphisms} $\C(x,y)$, for each pair of objects $x$ and $y$, equipped with a composition operation. The only coherence is the \emph{associativity law}.
Simply put, it is a category without identity structure.
A \emph{functor} $F\:\C \to \D$ between semicategories assigns to all objects $c$ in $\C$ an object $F(c)$ in $\D$ and to every morphism $f\:c \to c'$ in $\C$ a morphism $F(f)\:F(c) \to F(c')$ in $\D$ such that $F$ respects composition.
We denote by $\SemiCat$ the category of small semicategories and functors between them.

Generally, the notion of \emph{marking} or \emph{relative} captures the idea that certain things should be preserved \autocite{BK2012}.
For our interests, we would like to look at semicategories with certain morphisms that behave like identities and functors that preserve them.
Hence, this leads us to work with semicategories equipped with a class of morphisms. We call such objects \emph{relative semicategories}.
More precisely, a relative semicategory $\C$ consists of a carrier semicategory $\carrier{\C}$ and a wide subsemicategory $i_{\C}\:\marking{\C} \hookrightarrow \carrier{\C}$ (that is, they have the same class of objects) acting as the class of marked morphisms.
We denote by $\RelSemiCat$ the category of small relative semicategories and marking preserving functors.
\begin{rmk}\label{rmk:adjrsc}
	The category $\RelSemiCat$ can also be described as the category of wide monomorphisms of $\SemiCat$, namely $\Arr{\SemiCat}_{\Mbo}$ where $\Mbo$ is the class of bijective-on objects faithful functors.
	Thus, using this characterisation, a morphism $f\:\C \to \D$ in $\RelSemiCat$ can be viewed as a square
	\begin{center}
	\begin{tikzcd}
		\marking{C}
			\arrow[r, "\marking{f}"]
			\arrow[d, "i_{\C}"', hook'] &
		\marking{\D}
			\arrow[d, "i_{\D}", hook] \\
		\carrier{C}
			\arrow[r, "\carrier{f}"'] &
		\carrier{D}
	\end{tikzcd}
	\end{center}
\end{rmk}

The forgetful functors $\carrier{(\blank)},\marking{(\blank)}\:\RelSemiCat \to \SemiCat$ have left adjoints, respectively denoted by $(\blank)^{\flat}$ and $(\blank)^{\sharp}$, and form a quadruple adjunction
\begin{equation*}
	(\blank)^{\flat} \dashv \carrier{(\blank)}  \dashv (\blank)^{\sharp} \dashv \marking{(\blank)}.
\end{equation*}
The functor $(\blank)^{\sharp}\:\SemiCat \to \RelSemiCat$ marks all the morphisms and $(\blank)^{\flat}\:\SemiCat \to \RelSemiCat$ marks none of them.

\subsection{Monads with arities}\label{sec:monar}
The theory of monads with arities \cite{Weber2007,BMW2012} is the main tool we will use to study $\fDel$.
This concept is particularly useful when we study higher structures geometrically rather than algebraically (such as simplicial sets for categories), as it allows for the construction of nerve theorems and the analysis of Segal conditions.
In this section, we will first recall the concept of a monad and introduce tools such as induced monads and local right adjoints.
Then, we will define a monad with arities as a monad on a category equipped with a dense generator such that a composition preserves colimits related to the density.
We also recall the main results of \cite{Weber2007} and \cite{BMW2012} that we will use in the following sections.

\subsubsection{Cartesian monads}
A monad $(T,\mu,\nu)$ on a category $\E$ is an endofunctor $T\:\E \to \E$ equipped with a \emph{multiplication} $\mu\:TT \to T$ and a \emph{unit} $\nu\:\id_{\E} \to T$ satisfying coherences making it a monoid in the category of endofunctors of $\E$:
\begin{equation*}
	\mu\nu_T = \id_T = \mu T(\nu) \qquad \mu T(\mu) = \mu\mu_{T}.
\end{equation*}
The Eilenberg-Moore category $\EM{T}$ of a monad $T$ comes equipped with a free-forgetful adjunction $\adjunction{F^T}{\E}{\EM{T}}{U^T}$.
Recall that the category $\EM{T}$ consists of \emph{$T$-algebras} $(X,\alpha)$, that is objects $X$ of $\E$ with a morphism ${\alpha\:TX \to X}$ such that $\alpha T(\alpha) = \alpha \mu_X$ and $\alpha \nu_X = \id_X$, and the morphisms are morphisms in $\E$ commuting with the $T$-algebra structure.
The Kleisli category is the full subcategory of $\EM{T}$ consisting of the free $T$-algebras ${\mu_X\:TTX \to TX}$.
The previous adjunction restricts to a free-forgetful adjunction on the Kleisli category $\Kl{T}$.

We will say that a monad is \emph{cartesian} if $\E$ has pullbacks, $T$ preserves them and the natural transformations $\mu$ and $\nu$ are cartesian (i.e.\ all their naturality squares are cartesian).

If the category $\E$ has a terminal object $\tOb{}$, checking if $T$ is cartesian is easier:

\begin{lem}[{\cite[Remark~4.1.2b]{Leinster2004}}]\label{lem:charCartesianMonad}
	A monad $(T,\mu,\nu)$ is cartesian if and only if $\E$ has pullbacks, $T$ preserves them and the naturality squares for $\mu$ and $\nu$ at the terminal maps $!_X\:X\to \tOb{}$ are cartesian.\qed
\end{lem}

Let $\E$ be a category equipped with a wide class of morphisms $\M$.
We denote by $\Arr{\E}_{\M}$ the full subcategory of $\Arr{\E}$ consisting of maps in $\M$ and commutative squares between them.
Given a monad $(T,\mu,\nu)$ such that $T$ preserves the class $\M$, we define an induced monad $(T_{\M},\mu^{\M},\nu^{\M})$ on $\Arr{\E}_{\M}$.
The functor $T_{\M}\:\Arr{\E}_{\M} \to \Arr{\E}_{\M}$ maps arrows $f\in \M$ to $T(f)$ and similarly for squares.
The unit $\nu^{\M}\:\id_{\Arr{\E}_{\M}} \to T_{\M}$ and the multiplication $\mu^{\M}\:T_{\M}T_{\M} \to T_{\M}$ are defined at $f\in \M$ by the naturality squares of the unit $\nu\:\id_{\E} \to T$ and multiplication $\mu\:TT \to T$.
The naturality of $\nu^{\M}$ and $\mu^{\M}$ amount to the commutativity of cubes, which trivially follows from the naturality of $\nu$ and $\mu$.
Similarly, the monad laws for $T_{\M}$ follow from the laws for $T$.

\begin{prop}\label{prop:MonadClass}
	Let $\E$ be a category equipped with a wide class of morphisms $\M$ such that $\Arr{\E}_{\M}$ is closed under pullbacks, and let $T$ be a monad that restricts to $\M$.
	If $T$ is cartesian then so is $T_{\M}$.
\end{prop}
\begin{proof}
	Since $T$ is cartesian, $T_{\M}$ preserves pullbacks.
	As $\Arr{\E}_{\M}$ is closed under pullbacks, to prove that the monad $T_{\M}$ is cartesian we need to show that the unit $\nu^{\M}$ and the multiplication $\mu^{\M}$ are cartesian.
	This follows from the assumption that the natural transformations $\nu$ and $\mu$ are cartesian, as displayed by the following cubes corresponding to the naturality squares at $\alpha\: f \to f'$ in $\Arr{\E}_{\M}$:
	\begin{equation*}
	\begin{tikzcd}[row sep=scriptsize, column sep=scriptsize]
		& x
			\ar[dl, "f"']
			\ar[rr, "\nu_x"description]
			\ar[dd]
			\ar[ddrr, "\lrcorner" description, phantom, very near start]
		& & Tx
			\ar[dl, "Tf"]
			\ar[dd] \\
		y
			\ar[rr, "\nu_y" description, crossing over, near end]
			\ar[dd]
			\ar[ddrr, "\lrcorner" description, phantom, very near start]
		& & Ty \\
		& x'
			\ar[dl, "f'"']
			\ar[rr, "\nu_{x'}"description, near start]
		& & Tx'
			\ar[dl, "Tf'"] \\
		y'
			\ar[rr, "\nu_{y'}"description]
		& & Ty'
			\ar[from=uu, crossing over]
	\end{tikzcd}
	\qquad
	\begin{tikzcd}[row sep=scriptsize, column sep=scriptsize]
		& TTx
			\ar[dl, "TTf"']
			\ar[rr, "\mu_x" description]
			\ar[dd]
			\ar[ddrr, "\lrcorner" description, phantom, very near start]
		& & Tx
			\ar[dl, "Tf"]
			\ar[dd] \\
		TTy
			\ar[rr, "\mu_y" description, crossing over, near end]
			\ar[dd]
			\ar[ddrr, "\lrcorner" description, phantom, very near start]
		& & Ty \\
		& TTx'
			\ar[dl, "TTf'"']
			\ar[rr, "\mu_{x'}" description, near start]
		& & Tx'
			\ar[dl, "Tf'"] \\
		TTy'
			\ar[rr, "\mu_{y'}" description]
		& & Ty'
			\ar[from=uu, crossing over]
	\end{tikzcd}
	\end{equation*}
	The front and back squares are pullbacks, thus the corresponding squares are cartesian in $\Arr{\E}_{\M}$.
\end{proof}

\subsubsection{Local right adjoints}\label{sec:lra}
A functor $R\:\E \to \F$ is said to be a \emph{local right adjoint} (originally known as \emph{parametric right adjoint} \cite{Weber2007}) if the induced functor $R_X\:\E/X \to \F/RX$ has a left adjoint $L_X$, for each object $X$.

This notion can be described differently using so-called \emph{$R$-generic} maps.
A morphism $g\:A \vbto RX$, depicted by a vertical bar, is said to be $R$-generic if, given $\alpha$, $\beta$ and $\gamma$ making the diagram
\begin{center}
	\begin{tikzcd}
		A
		\ar[r, "\alpha"]
		\ar[d, mid vert, "g"'] &
		RX'
		\ar[d, "R\gamma"] \\
		RX
		\ar[r, "R\beta"']
		\ar[ur, "R\delta" description, dotted] &
		RY
	\end{tikzcd}
\end{center}
commute, there is a unique $\delta\:X \to X'$ such that $R\delta\,g=\alpha$ and $\beta=\gamma \delta$.
Notice that the lower triangle commutes before applying the functor.
The morphism $\delta\:X \to X'$ obtained by $R$-genericness is called an \emph{$R$-fill} in \autocite{Weber2007}; we will call it an \emph{$R$-filler} of the square.
We will say that a monad $(T,\mu,\nu)$ is \emph{strongly cartesian} if it is cartesian and $T$ is a local right adjoint.

\begin{prop}[{\cite[Lemma~2.7]{BMW2012}}]\label{prop:lrac}
	For a functor $R\:\E \to \F$ the following are equivalent:
	\begin{enumerate}[label=\arabic*.,ref=\theprop (\arabic*)]
		\item\label{prop:lrac1} R is a local right adjoint;
		\item\label[prop]{prop:lrac2} each $f\:A \to RY$ factors as $A \xvbto{g} RX \xrightarrow{Rh} RY$, where $g$ is $R$-generic.
	\end{enumerate}
	If $\E$ has a terminal object, then these conditions can be equivalently reformulated as:
	\begin{enumerate}[label=\arabic*'.]
		\item the functor $R_{\tOb{\E}}\:\E \to \F/R\tOb{\E}$ has a left adjoint;
		\item each $f\:A \to R\tOb{\E}$ factors as $A \xvbto{g} RX \xrightarrow{R!_X} R\tOb{\E}$, where $g$ is $R$-generic.\qed
	\end{enumerate}
\end{prop}

Using the previous characterisation, it is not difficult to observe that a local right adjoint $R\:\E \to \F$ induces another local right adjoint $R_*\: \E^{\I} \to \F^{\I}$ on the categories of diagrams, where $\I$ is a category.

\begin{prop}\label{prop:genind}
	If $R\:\E \to \F$ is a local right adjoint, then so is the induced functor $R_*\:\E^{\I} \to \F^{\I}$ for any category $\I$.
\end{prop}
\begin{proof}
  Let $\rho\:D \to R_*D'$ be a natural transformation with $D\:\I \to \F$ and $D'\:\I \to \E$ diagrams.
  For each $i \in \I$, we consider the $R$-generic factorisation $R\freem{\rho}_i\,\genm{\rho}_i\: D_i \vbto RX_i \to RD'_i$ of $\rho_i$, as in \cref{prop:lrac2}.
  We extend the assignment $i \mapsto X_i$ to a functor $X : \I \to \E$ by defining $Xj$ for a morphism $j : i_0 \to i_1$ in $\I$ to be the $R$-filler for the left diagram below.
  This filler also makes the right diagram commute, which shows the naturality of $\genm{\rho}$ and, since the right square commutes before applying $R$, also the naturality of $\freem{\rho}$.
	\begin{equation}\label{diag:xfung}
	\begin{tikzcd}[column sep=1.5cm]
		D_{i_0}
			\ar[r,"\genm{\rho}_{i_1}Dj"]
			\ar[d,"\genm{\rho}_{i_0}"',mid vert] &
		RX_{i_1}
			\ar[d,"R\freem{\rho}_{i_1}"] \\
		RX_{i_0}
			\ar[ur,"RX_j" description,dotted]
			\ar[r,"RD'_j\,R\freem{\rho}_{i_0}"'] &
		RD'_{i_1}
	\end{tikzcd}
	\hspace{1cm}
	\begin{tikzcd}
		D_{i_0}
			\ar[d, "Dj"']
			\ar[r, "\genm{\rho}_{i_0}", mid vert] &
		RX_{i_0}
			\ar[r, "R\freem{\rho}_{i_0}"]
			\ar[d, "RX_j"] &
		RD'_{i_0}
			\ar[d, "RD'j"] \\
		D_{i_1}
			\ar[r, "\genm{\rho}_{i_1}"', mid vert] &
		RX_{i_1}
			\ar[r, "R\freem{\rho}_{i_1}"'] &
		RD'_{i_1}
	\end{tikzcd}
	\end{equation}
    Thus, we can factor $\rho$ as $R\freem{\rho}\,\genm{\rho}\: D \to R_*X \to R_*D'$.
	We now show that $\genm{\rho}\:D \to R_*X$ is $R_*$-generic. Suppose $\alpha$, $\beta$, and $\gamma$ make the square commute
	\begin{equation}\label{diag:r*gen}
	\begin{tikzcd}
		D
			\ar[r, "\alpha"]
			\ar[d, "\genm{\rho}"'] &
		R_*A
			\ar[d, "R_*\gamma"] \\
		R_*X
			\ar[r, "R_*\beta"'] &
		R_*B
	\end{tikzcd}
	\end{equation}
	Since the $\genm{\rho}_i$'s are $R$-generic, we get $R$-fillers $\delta_i\:X_i \to A_i$ making the relevant diagram commute.
	The $\delta_i$'s assemble into a unique natural transformation $\delta\:X \to A$ satisfying the equations for $R_*$-genericness.
	In fact, consider the following diagrams of solid arrows:
	\begin{equation*}
	\begin{tikzcd}[column sep=1cm, row sep=1cm]
		D_{i_0}
			\ar[r, "\alpha_{i_0}"]
			\ar[d, "\genm{\rho}_{i_0}"', mid vert] &
		RA_{i_0}
			\ar[r, "RAj"] &
		RA_{i_1}
			\ar[d, "R\gamma_{i_1}"] \\
		RX_{i_0}
			\ar[r, "RXj"']
			\ar[ur, "R\delta_{i_0}" description, bend left=10]
			\ar[urr, dotted] &
		RX_{i_1}
			\ar[r, "R\beta_{i_1}"']
			\ar[ur, "R\delta_{i_1}" description, bend right=10] &
		RB_{i_1}
	\end{tikzcd}
	\end{equation*}
	We observe that the compositions $Aj\,\delta_{i_0}$ and $\delta_{i_1}Xj$ satisfy the equations of $R$-fillers of the outer commutative square by construction.
	The commutativity of the outer square can be checked using the naturality of $\genm{\rho}$ at $j$ from \eqref{diag:xfung}, the square \eqref{diag:r*gen} at $i_1$ and the naturality of $\alpha$ at $j$.
	Thus, by $R$-genericness of $\genm{\rho}_{i_0}$, we have the equality $Aj\,\delta_{i_0}=\delta_{i_1}Xj$ witnessing the naturality of $\delta$.
	Finally, $\delta$ is pointwise unique and satisfies the equations $\gamma\delta=\beta$ and $R_*\delta\genm{\rho}=\alpha$ pointwise, and thus is an $R_*$-filler of the square \eqref{diag:r*gen}.
\end{proof}

As expected by the left adjointness, genericness interacts very nicely with colimits, as illustrated in the remaining results of this section.
In particular, they will be used in the study of the free relative semicategory monad in \cref{sec:frelsem}.

\begin{lem}\label{lem:ColimPresGen}
	Let $R\:\E \to \F$ be a functor and $D\:\I \to \F$ and $D'\:\I \to \E$ diagrams, and suppose their colimits exist.
	Given a natural transformation $\rho\:D \vbto RD'$ such that each component is $R$-generic, the canonical morphism $\hat{\rho}\:\colim_{\I} D_i \vbto R\colim_{\I} D'_i$ is $R$-generic.
\end{lem}
\begin{proof}
	Suppose $\alpha$, $\beta$ and $\gamma$ make the right square commute.
	\begin{equation*}
	\begin{tikzcd}[row sep=2cm]
		D_i
			\ar[d, "\rho_i"', mid vert]
			\ar[r, "d_i"] &
		\colim_{\I} D
			\ar[r, "\alpha"] &
		RX
			\ar[d, "R\gamma"] \\
		R D'_i
			\ar[r, "Rd'_i"']
			\ar[urr, "R\delta_i" description, dotted, near start] &
		R\colim_{\I} D'
			\ar[r, "R\beta"']
			\ar[ur, "R\delta" description, dotted]
			\ar[from=u, "\hat{\rho}" description, crossing over,near start] &
		RY
	\end{tikzcd}
	\end{equation*}
	Since $\rho_i$ is $R$-generic, for each $i\in \I$, we have a unique map $\delta_i\:D'_i \to X$ such that
	$R\delta_i\,\rho_i=\alpha d_i$ and $\gamma\delta_i=\beta d'_i$.
	Moreover, for any map $j\:i_0 \to i_1$ in $\I$, the composite $\delta_{i_1}D'j$ is also an $R$-filler of the outer square at $i = i_0$.
	For the lower triangle, we have $\gamma\delta_{i_1}D'j=\beta d'_{i_1}D'j=\beta d'_{i_0}$, and the equation for the upper triangle is similarly checked.
	Therefore, by uniqueness of $\delta_{i_0}$ we have $\delta_{i_1}D'j= \delta_{i_0}$.
	Consequently, the universal property of $\colim_{\I} D'$ gives us a map $\delta\:\colim_{\I} D' \to X$ and, by uniqueness, we get that $\gamma\delta=\beta$ and $R\delta\,\hat{\rho}=\alpha$.
\end{proof}

We now prove an analogous result for $R$-generic factorisations instead of $R$-generic maps.

\begin{lem}\label{lem:ColimPresGenFact}
	Let $R\:\E \to \F$ be a functor with $\E$ a cocomplete category, and let ${D'\:\I\to \E}$ and $D\:\I\to \F$ be diagrams.
	Suppose we have a natural transformation $\rho\:D \to R_*D'$ such that the component maps $\rho_i\:D_i \to RD'_i$ all admit $R$\nobreakdash-generic factorisations, and suppose $\colim_{\I}D$ exists.
	Then the canonical map $f\:\colim_{\I} D \to R\colim_{\I} D'$ has an $R$-generic factorisation.
\end{lem}
\begin{proof}
	Define the $R_*$-generic factorisation $D \xvblongto{\genm{\rho}} R_*X \xlongrightarrow{R\freem{\rho}} R_*D'$ of $\rho$ as in the proof of \cref{prop:genind}, and construct the canonical maps $\hat{f}\:\colim_{\I} D \to R\colim_{\I} X$ and $\tilde{f}\:\colim_{\I} X \to \colim_{\I} D'$ as displayed in the following diagram.
	\begin{equation}\label{diag:Rgefact}
	\begin{tikzcd}[row sep=1cm, column sep=1.5cm]
		D_i
			\arrow[r, "\hat{\rho}_i" description, near start, mid vert]
			\arrow[d, "d_i"' description]
			\arrow[rr, "\rho_i" description, bend left] &
		RX_i
			\arrow[r, "R\tilde{\rho}_i" description]
			\arrow[d, "Rx_i" description] &
		RD'_i
			\arrow[d, "Rd'_i" description] \\
		\colim_{\I}D
			\arrow[r, "\hat{f}" description, near start, mid vert, dotted]
			\arrow[rr, "f" description, bend right] &
		R\colim_{\I} X
			\arrow[r, "R\tilde{f}" description, dotted] &
		R\colim_{\I}D'
	\end{tikzcd}
	\end{equation}
	Let us show that the bottom composite is an $R$-generic factorisation of $f$.
	The $R$-genericness of $\hat{f}\:\colim_{\I}D \vbto R\colim_{\I}X$ follows from \cref{lem:ColimPresGen} applied to $\hat{\rho}\:D \vbto RX$.
	By the universal property of $\colim_{\I}D$, it suffices to check that $R\tilde{f}\,\hat{f}\,d_i = f\,d_i$ to show that the bottom triangle in \eqref{diag:Rgefact} commutes.
	But, since the left, right and outer squares commute by definition of $\hat{f}$, $\tilde{f}$ and $f$ respectively, and the top triangle is a factorisation, the equality is verified.
\end{proof}

\begin{rmk}\label{rmk:cocom}
	In \cref{lem:ColimPresGenFact}, we assume for convenience that $\E$ is cocomplete.
	The proof only requires that the colimits of $D$, $D'$ and $X$ exist.
\end{rmk}

Let $\E$ be a category.
A full subcategory $\A$ and the inclusion functor $i_{\A}\:\A \to \E$ are called \emph{dense} if $\A$ is small and the associated \emph{nerve functor}
\begin{equation*}
	\Nerve_{\A}\:\E \to \yhwidehat{\A}, \quad x \mapsto \E(i_{\A},x)
\end{equation*}
is fully faithful.
The dense subcategory $\A$ is also called a \emph{dense generator}.

\begin{rmk}\label{rmk:dencol}
	Recall that density is equivalent to requiring that any object $X$ in $\E$ is the colimit over the comma category $\A/X$ given by the composition $i_{\A}\pi_{X}$, where $\pi_{X}\:\A/X \rightarrow \A$ is the domain projection.
	Let us call such colimits \emph{canonical $\A$-colimits}.
\end{rmk}

The following improves \cref{prop:lrac} in case we have a dense subcategory and was proved for presheaf categories with representables as the dense subcategory in \cite[Proposition 2.10]{Weber2007}.

\begin{prop}\label{prop:LRAdenseGen}
	Let $\E$ be a cocomplete category with a terminal object $\tOb{\E}$, and $\F$ a category with a dense generator $\A$.
	A functor $R\:\E \to \F$ is a local right adjoint if and only if any map $f\:A \to R \tOb{\E}$, with $A \in \A$, has an $R$-generic factorisation.
\end{prop}
\begin{proof}
	One direction is straightforward using \cref{prop:lrac}.
	Conversely, let $f\:X\to R\tOb{\E}$ be a map in $\F$.
	Again, by \cref{prop:lrac} it suffices to show that $f$ has an $R$-generic factorisation.
	By density, $X \cong \colim_{\A/X} i_{\A}\pi_X$. Consider the natural transformation
	\[
	i_{\A}\pi_X \longrightarrow R\tOb{\E}
	\]
	whose component at $(A\to X)$ is the composite
	$A\to X \xrightarrow{f} R\tOb{\E}$.  By the hypothesis, all its components admit
	$R$-generic factorisations.  Applying \cref{lem:ColimPresGenFact} to the
	constant diagram at $\tOb{\E}$ gives an $R$-generic factorisation
	\[
	X \xrightarrow{g} RY \xrightarrow{Rh} R\!\colim_{\A/X}\tOb{\E}
	\]
	of the canonical map $X\to R\!\colim_{\A/X}\tOb{\E}$ induced by the above
	cocone.
	Let
	$c\:\colim_{\A/X}\tOb{\E}\to\tOb{\E}$ be the unique map.
	Then $f=R(c)Rh\,g$, since this equality holds after precomposition with every
	$A\to X$ in the canonical dense cocone.
	Hence $g$ is the generic part of an $R$-generic factorisation of $f$.  By \cref{prop:lrac}, $R$ is a local right
	adjoint.
\end{proof}

\subsubsection{The theory of monads with arities}\label{sec:mwa}
A monad $(T,\mu,\nu)$ on a category $\E$ with a dense generator $i_{\A}\:\A\hookrightarrow \E$ is called a \emph{monad with arities} if the composition $\Nerve_{\A}T$ preserves canonical $\A$\nobreakdash-colimits \cite{Weber2007,BMW2012}.
In other words, $T$ is equivalent to an $i_{\A}$-relative monad \cite{ACU2015}.
The category $\Theta_T$ is defined by factoring the composition $F^T i_{\A}\:\A \to \EM{T}$ into a bijective-on-objects functor $j_T\:\A \to \Theta_T$ followed by a fully faithful functor $i_T\:\Theta_T \to \EM{T}$.
More precisely, the category $\Theta_T$ can be characterised as the full subcategory spanned by the free $T$-algebras over $\A$.
For a monad $T$ with arities $\A$, the category $\Theta_T$ is dense in $\EM{T}$:

\begin{thm}[Abstract nerve theorem {\autocite[Theorem 1.10]{BMW2012}}]\label{thm:nerve}
	Let $T$ be a monad with arities $\A$.
	The nerve $\Nerve_T\:\EM{T} \to \yhwidehat{\Theta_T}$ associated to the full inclusion $i_T$ is fully faithful, and its essential image is spanned by the presheaves whose restriction along $j_T$ belongs to the essential image of $\Nerve_{\A}$.\qed
\end{thm}

Let $T$ be a strongly cartesian monad on a finitely complete category $\E$ with a dense generator $\A$.
The category of \emph{canonical arities} $\Ar{\A}{T}$ is defined as the essential image of the composition
\begin{equation*}
	\A/T\tOb{\E} \xlongrightarrow{i_{\A}/T\tOb{\E}} \E/T\tOb{\E} \xlongrightarrow{L_{\tOb{\E}}} \E,
\end{equation*}
where $L_{\tOb{\E}}\:\E/T\tOb{\E} \to \E$ is left adjoint to $T_{\tOb{\E}}\:\E \to \E/T1$.
A dense generator $\A$ of $\E$ is \emph{$T$-generically closed} if for any $T$-generic morphism $A \to TB$ with $A$ in~$\A$, there is an object isomorphic to $B$ which belongs to $\A$.

\begin{thm}[{\cite[Theorem 2.9]{BMW2012}}]\label{thm:scmag}
	Let $T$ be a strongly cartesian monad on a finitely complete category $\E$ with a dense generator $\A$.
	The category of canonical arities $\Ar{\A}{T}$ is $T$-generically closed and provides $T$ with arities $\Ar{\A}{T}$.\qed
\end{thm}

Let $\E$ be a category with a dense generator $\A$, a \emph{theory $(\Theta,j)$ with arities $\A$} on $\E$ consists of a bijective-on-objects functor $j\:\A \to \Theta$ such that the induced monad $j^*j_!$ on $\yhwidehat{\A}$ preserves the essential image of the nerve $\Nerve_{\A}\:\E \to \yhwidehat{\A}$.
A~morphism $g$ in $\Theta$ is called \emph{free} if it is in the image of $j$ and \emph{generic} if for each factorisation $g=j(f)g'$ through a free morphism, the latter is invertible in $\A$, see \cite[Section 3]{BMW2012}.

The category $\Theta_T$ defined above is an example of a theory with arities \cite[Proposition 3.2]{BMW2012}, and we will refer to it as the \emph{theory associated to the monad $T$}.

\begin{rmk}[{\cite[Theorem 3.4]{BMW2012}}]
	There is an equivalence between the category of monads with arities $\A$ and the category of theories with arities $\A$.
\end{rmk}

A theory $(\Theta,j)$ with arities $\A$ on $\E$ is called \emph{homogeneous} if $j$ is faithful and $\Theta$ admits a generic-free factorisation system, see just after Definition 3.9 in \autocite{BMW2012}.

\begin{thm}[{\cite[Theorem~3.10]{BMW2012}}]\label{thm:homog}
	Let $\E$ be a finitely complete category with a dense generator $\A$. For any strongly cartesian monad $T$, the associated theory $(\Theta_T,j_T)$ with arities $\Ar{\A}{T}$ is homogeneous.\qed
\end{thm}


%% file: sec_fatDelta.tex
\section{Presentations of fat Delta}\label{sec:fDeldef}
The category $\fDel$ was first introduced by J.\ Kock in \cite{Kock2006} through the concept of  ``coloured semiordinals'' but can also be viewed in various ways, as mentioned by J.\ Kock in the same work.
In this section, we review two descriptions introduced in \cite{Kock2006}.
We present the original description, useful for intuition, in terms of relative semiordinals (\cref{sec:rso}).
We then give another one, more practical, also defined in \cite{KS2017,Pao2025} in terms of a subcategory of the category of epimorphisms in $\Delta$ (\cref{sec:epifdel}).

In this paper, we make use of the intuition of the relative semiordinals, but we exclusively work with the definition of \cref{sec:epifdel}.
As a consequence of the results of \cref{sec:frelsem}, we will formally prove in \cref{thm:fDeldesc} that these two presentations are isomorphic.

The current section also introduces tools and notation to simplify the manipulation of the objects of $\fDel$, and two classes of morphisms (active and inert) which will be proven to form an orthogonal factorisation system on $\fDel$ in \cref{sec:nervethm}.

\subsection{Relative semiordinals}\label{sec:rso}
In \cite{Kock2006}, J.\ Kock focuses on a conceptual approach to $\fDel$ by defining it as the full subcategory of $\RelSemiCat$ consisting of finite non-empty ``coloured semiordinals''. We present this description here to provide intuition for the objects and morphisms of $\fDel$.

A \emph{semiordinal} is a semicategory associated with a total strict order relation. In particular, the relation is not reflexive, so there are no identity morphisms.
The category of finite non-empty semiordinals is naturally identified with the semisimplex category $\Delta_+$ (i.e.\ the simplex category $\Delta$ restricted to its monomorphisms), as the lack of identities allows only monomorphisms between semiordinals.
Thus, a \emph{relative semiordinal} is a semiordinal equipped with a wide subsemicategory.
The category $\fDel$ is defined as the category of finite non-empty relative semiordinals and marking-preserving (i.e.\ colour-preserving in~\cite{Kock2006}) functors between them.

\begin{rmk}
	Ordinals can also be described as the free categories on linearly ordered graphs.
	Similarly, relative semiordinals are the free relative semicategories on linearly ordered relative graphs.
	We will make this precise in \cref{sec:frelsem}.
\end{rmk}

Since $\fDel$ is characterised here as a full subcategory of $\RelSemiCat$, this description induces a fully faithful functor $\mathsf{rso}\: \fDel \hookrightarrow \RelSemiCat$.
This helps with intuition as we can represent objects of $\fDel$ as a string of arrows, some of which are marked, and maps as inclusions that respect the marking.
For instance, in $\fDel$ we have the usual semiordinals
\begin{equation}
	\langle\bullet\rangle, \quad 
	\langle\bullet \longrightarrow \bullet\rangle, \quad 
	\langle\bullet \longrightarrow \bullet \longrightarrow \bullet\rangle, \quad
	\langle\bullet \longrightarrow \bullet \longrightarrow \bullet \longrightarrow \bullet\rangle, \quad \ldots
\end{equation}
as well as partially marked semiordinals
\begin{equation}\label{diag:lrneut}
	\langle\bullet \mlgar \bullet \longrightarrow \bullet\rangle \quad \text{or} \quad
	\langle\bullet \longrightarrow \bullet \mlgar \bullet\rangle
\end{equation}
and fully marked semiordinals
\begin{equation}
	\langle\bullet \mlgar \bullet\rangle, \quad 
	\langle\bullet \mlgar \bullet \mlgar \bullet\rangle, \quad
	\langle\bullet \mlgar \bullet \mlgar \bullet \mlgar \bullet\rangle, \quad \ldots,
\end{equation}
where the red arrows encode the marking.
Morphisms between them are strictly monotone injections that map red arrows to red arrows.
This is a very intuitive way to think about $\fDel$, but it is not the most practical for computations.

\subsection{Epimorphisms in \texorpdfstring{$\Delta$}{Delta}}\label{sec:epifdel}
We now introduce the primary definition of $\fDel$ used in this paper.
This characterisation is more suitable for our purposes and studies.
Henceforth, we will use the notation $\fDel$ exclusively to refer to the category described in this section.

The category $\fDel$ is defined as $(\Arr{\Delta}_{-})_{\Mono}$, where $\Delta_-$ is the restriction of the simplex category to its epimorphisms.
Therefore, the objects are epimorphisms $\eta\:[m] \twoheadrightarrow [n]$ in $\Delta$, i.e.\ composition of degeneracies, and morphisms $f\:\eta_0 \to \eta_1$ are commutative squares in $\Delta$
\begin{equation}\label{diag:mofDel}
	\begin{tikzcd}
		{[m_0]}
			\ar[r, "\ftop{f}", hook]
			\ar[d, "\eta_0"', twoheadrightarrow]
		& {[m_1]}
			\ar[d, "\eta_1", twoheadrightarrow] \\
		{[n_0]}
			\ar[r, "\fbot{f}"']
		& {[n_1]}
	\end{tikzcd}
\end{equation}
whose vertical arrows are epimorphisms and whose horizontal top arrow is a monomorphism.
Intuitively, the domain of the epimorphisms corresponds to the length of the semiordinal and the marked edges are the ones that are sent to the identity.
For instance, objects as in \eqref{diag:lrneut} can be represented diagrammatically by the epimorphisms
\begin{equation*}
\begin{tikzcd}[column sep=0.2cm]
	\bullet
		\ar[d,mapsto]
		\ar[rr, thick, dash]
	&& \bullet
		\ar[dr,mapsto]
		\ar[rr, thick, dash, red]
	&& \bullet
		\ar[dl,mapsto] \\
	\bullet
		\ar[rrr, thick, dash]
	&&& \bullet
	&
\end{tikzcd}
\quad\text{and}\quad
\begin{tikzcd}[column sep=0.2cm]
	\bullet
		\ar[dr,mapsto]
		\ar[rr, thick, dash, red]
	&& \bullet
		\ar[dl,mapsto]
		\ar[rr, thick, dash]
	&& \bullet
		\ar[d,mapsto] \\
	& \bullet
		\ar[rrr, thick, dash]
	&&& \bullet
\end{tikzcd}
\end{equation*}
where a marked edge is collapsed to a single vertex.

\begin{rmk}\label{rmk:bodet}
	Given a map $f\:\eta_0 \to \eta_1$ in $\fDel$, as in \eqref{diag:mofDel} above, the bottom morphism $\fbot{f}$ is entirely determined by $\ftop{f}$, $\eta_1$ and any section $s$ of $\eta_0$.
	Indeed, we have $\fbot{f} = \eta_1\ftop{f}s$.
\end{rmk}

The restriction of the domain and codomain functors provides projections
\begin{equation*}
	\overline{\pi}\:\fDel \to \Delta_+ \quad \text{and} \quad \fat{\pi}\:\fDel \to \Delta,
\end{equation*}
respectively.
The functor $\fat{\pi}$ is significant as it induces a Dwyer-Kan equivalence between the homotopical categories $(\Delta,\Iso)$ and $(\fDel,\fVert)$, where $\fVert$ is the class of morphisms mapped to identities in $\Delta$ by $\fat{\pi}$ \cite{Sattler2017}.
In addition, $\fat{\pi}$ is an \emph{ambifibration}, which endows it with some nice lifting properties.
The literature~on ambifibrations is quite limited but, for more details, we refer to \cite[Section 5.3]{CSZ2026} and to the unpublished note \cite{Sattler2017} or to Kock's comment on $n$-Category Caf\'{e} \cite{KockCom}.

There is a `horizontal' inclusion
\begin{equation*}
	(\blank)^{\flat}\:\Delta_+ \hookrightarrow \fDel,
\end{equation*}
which maps a semiordinal $[n]$ to the corresponding identity morphism $\id_{[n]}$ and a monomorphism $[n] \hookrightarrow [n']$ to the evident square.
The objects in the image of the horizontal inclusion intuitively correspond to the relative semiordinals with nothing marked. The composite
\begin{equation*}
	\Delta_+ \xhookrightarrow{(\blank)^{\flat}} \fDel \xtwoheadrightarrow{\fat{\pi}} \Delta
\end{equation*}
factors the canonical inclusion $\Delta_+ \hookrightarrow \Delta$.
It is in this setting that one may consider $\fDel$ as a direct replacement of $\Delta$.
Indeed, $\fDel$ is a direct category and $(\fDel,\fVert)$ has the desired structure since $\fat{\pi}$ induces a Dwyer-Kan equivalence with $\Delta$.
This former property is of particular relevance in the context of HoTT-related frameworks, as it allows us to inductively construct the type of presheaves over the direct category, as discussed in \cref{sec:intrott}.
We also have a `vertical' inclusion
\begin{equation*}
	(\blank)^{\sharp}\:\Delta_+ \hookrightarrow \fDel
\end{equation*}
which maps a semiordinal $[n]$ to the canonical map $[n] \twoheadrightarrow [0]$ and a monomorphism $[n] \hookrightarrow [n']$ to the obvious square with the bottom map being the identity $\id_{[0]}$.
The objects in the image of this inclusion intuitively correspond to the relative semiordinals with everything marked.

\subsection{Active and inert morphisms}\label{sec:activeinert}
As we will show in \cref{thm:fDelhomo,cor:actinfDel} using the theory of monads with arities and by giving a direct proof, the category $\fDel$ also possesses a special orthogonal factorisation system $(\fDel_a,\fDel_0)$ called \emph{active-inert}.
This type of factorisation system usually appears in the context of the \emph{Segal conditions} and \emph{nerve functors} as it makes it possible to express the Segal conditions in terms of a restriction along the right class~\cite{BMW2012}.
This is thus an important tool to have to study higher categories. It was first introduced as a general abstract notion in category theory by M.\ Weber in \cite{Weber2004,Weber2007} after deriving ideas from A.\ Joyal and C.\ Berger \cite{Joyal1986,Ber2002}. It was further studied by C.\ Berger, P.A.\ Melli\`es and M.\ Weber \cite{BMW2012}, and was then called \emph{generic-free}.
The more recent terminology active-inert is due to J.\ Lurie \cite{LurieHA}.
This new terminology has been extensively used in recent publications \cite{CH2021,Ber2023,HK2022,ShapiroPhD} and especially in the context of \emph{decomposition spaces} in \cite{GKT2018a} and follow-ups.
In the simplicial case, the active and inert morphisms have a very neat characterisation: they correspond to endpoint and distance-preserving maps, respectively.
In $\fDel$ the characterisation is similar, with the difference that the maps respect the marking.
The classes $\fDel_a$ and $\fDel_0$ are given by
\begin{equation}\label{diag:actin}
\left\{
\begin{tikzcd}[column sep=1.2cm]
	{[m]}
		\ar[r, "\text{active}", hookrightarrow]
		\ar[d, twoheadrightarrow]
		\ar[dr, phantom, very near end, "{\ulcorner}"]
	& {[m']}
		\ar[d, twoheadrightarrow] \\
	{[n]}
		\ar[r]
	& {[n']}
\end{tikzcd}
\right\}
\quad \text{and} \quad
\left\{
\begin{tikzcd}[column sep=1.2cm]
	{[m]}
		\ar[r, "\text{inert}", hookrightarrow]
		\ar[d, twoheadrightarrow] &
	{[m']}
		\ar[d, twoheadrightarrow] \\
	{[n]}
		\ar[r] &
	{[n']}
\end{tikzcd}
\right\}
\end{equation}
respectively.
Note that active maps in $\fDel$ are required to form pushout squares in $\Delta$, which exist by \cite[Corollary 3.3]{CFPS2023}. 
This requirement ensures the uniqueness of the factorisation.
Indeed, the property of preserving endpoints must restrict to the decomposition of the codomain into intervals of consecutively marked or unmarked edges.
Otherwise, the following kind of issue arises:
\begin{equation*}
\begin{tikzcd}[column sep=0.2cm, font=\scriptsize, nodes={inner sep=1pt}]
	\langle 0
		\ar[rr, dash, thick, red] &&
	1
		\ar[rrrrd, "{(\delta_2,\delta_1)}"', bend right]
		\ar[rr, thick, dash] &&
	2\rangle
		\ar[rrrrrrrrrr, "{(\delta_2,\id_{[1]})}"] &&&&&&&&&&
	\langle0
		\ar[rr, dash, thick, red] &&
	1
		\ar[rr, dash, thick, red] &|[emptyNode]|\ar[from=llllld, "{(\id_{[3]},\sigma_0)}"', bend right]& 
	\bullet
		\ar[rr, thick, dash] &&
	2\rangle \\
	&&&&&&
	\langle 0
		\ar[rr, dash, thick, red] &&
	1
		\ar[rr, thick, dash] && 
	\bullet
		\ar[rr, thick, dash] &&
	2\rangle &&&&&&&&
\end{tikzcd}
\end{equation*}
The morphism $(\delta_2,\id_{[1]})$ pictured above, with the degeneracy ${\sigma_0\:[2] \twoheadrightarrow [1]}$ as its domain and ${\sigma_0\sigma_0\:[3] \twoheadrightarrow [1]}$ as its codomain, is an example of a morphism in $\fDel$ that has an active map at the top, namely ${\delta_2\:[2] \hookrightarrow [3]}$, but is not a pushout square.
It can be factored through the object ${\sigma_0\:[3] \twoheadrightarrow [2]}$ via the active map $(\delta_2,\delta_1)$ followed by the inert map $(\id_{[3]},\sigma_0)$.

\begin{rmk}\label{rmk:actbot}
	Given an active morphism of $\fDel$, the corresponding bottom map is active in $\Delta$, since left classes are preserved by pushouts.
	It is also a monomorphism by the following lemma.
\end{rmk}

\begin{lem}\label{lem:moprep}
	In $\Delta$, monomorphisms are preserved by pushouts along epimorphisms.
\end{lem}
\begin{proof}
	Since, in $\Delta$, monomorphisms are compositions of faces and epimorphisms are compositions of degeneracies, by pushout pasting it is enough to show that the pushout of a face map along a degeneracy is a monomorphism.
	This is a direct consequence of \cite[Corollary 3.11 (b)]{CFPS2023}.
\end{proof}

\begin{rmk}\label{rmk:mcontract}
	Given an inert morphism in~$\fDel$, the commutativity of the square forces the bottom arrow $f$ to have the following property: $f(i+1) \leq f(i) + 1$, for $0\leq i \leq n-1$.
\end{rmk}

\cref{lem:3.7} records an observation about maps in $\Delta$ having this property that is not required for this paper but may be of independent interest.
For that, we introduce the \emph{relative $\vpro$\nobreakdash-product}.
Recall from~\cite[Definition 2.6]{CFPS2023} that the $\vpro$-product of $[n]$ and $[n']$ is defined by the pushout
\begin{equation*}
\begin{tikzcd}
	{[0]}
		\ar[r, "n"]
		\ar[d, "0"']
		\ar[dr, phantom, very near end, "{\ulcorner}"]
	& {[n]}
		\ar[d] \\
	{[n']}
		\ar[r]
	& {[n] \vpro [n']}
\end{tikzcd}
\end{equation*}
In particular, $[n] \vpro [n'] =[n+n']$.
As the $\vpro$-product is functorial on active morphisms of $\Delta$, \cite[Definition 2.6]{CFPS2023}, it induces an operation on objects of $\fDel$ (epimorphisms in $\Delta$ are always active), denoted by $\vpro$ as well, and called the relative $\vpro$-product.
Given two objects $\eta\:[m] \twoheadrightarrow [n]$ and $\eta'\:[m'] \twoheadrightarrow [n']$ of $\fDel$, the relative semiordinal
\begin{equation*}
  \eta\vpro\eta'\:[m]\vpro[m'] \twoheadrightarrow [n]\vpro[n']
\end{equation*}
is obtained by gluing $\eta'$ to the maximal element of $\eta$.
Intuitively, the relative $\vpro$-product of two objects $\eta$ and $\eta'$ in $\fDel$ connects the last vertex of $\eta$ directly to the first vertex of $\eta'$. For instance, in the case of \eqref{diag:lrneut}, we have
\begin{equation*}
	\langle\bullet \longrightarrow \bullet \mlgar \bullet\rangle \vpro \langle\bullet \mlgar \bullet \longrightarrow \bullet\rangle =
	\langle\bullet \longrightarrow \bullet \mlgar \bullet \mlgar \bullet \longrightarrow \bullet\rangle.
\end{equation*}

\begin{rmk}\label{rmk:indec}
By decomposing the codomain, any inert map $[m] \hookrightarrow [m']$ in~$\Delta$ is a canonical inclusion of the interval $[m] \hookrightarrow [a]\vpro[m]\vpro[b]$.
Moreover, this extends to inert maps in $\fDel$.
Specifically, if $\eta \to \eta'$ is inert, it is given by a commutative square
\begin{equation*}
\begin{tikzcd}
	{[m]}
	\ar[r, hook]
	\ar[d, "\eta"', twoheadrightarrow]
	& {[a]\vpro[m]\vpro[b]}
	\ar[d, "\eta'", twoheadrightarrow] \\
	{[n]}
	\ar[r]
	& {[a]\vpro[m']\vpro[b]}
\end{tikzcd}
\end{equation*}
where $m'\leq m$ and the image of the bottom arrow restricts to $[m']$.
\end{rmk}

\begin{lem}\label{lem:3.7}
	A map $f\:[m] \to [n]$ in $\Delta$ has the property $f(i+1) \leq f(i) + 1$ for $0\leq i \leq m-1$ if and only if its epi-mono factorisation coincides with its active-inert factorisation.
\end{lem}
\begin{proof}
	Suppose $f$ has the property, then the image $[k]$, with $k\leq m$, is an interval in $[n]$ since either $f(i+1)=f(i)$ or $f(i+1)=f(i)+1$ (by order preservation).
	In particular, the epi-mono factorisation of $f$ is of the form
	\begin{equation}\label{diag:epimono}
		[m] \twoheadrightarrow [k] \hookrightarrow [a] \vpro [k] \vpro [b].
	\end{equation}
	The first map is automatically active, and the second is inert (see \cref{rmk:indec}).
	Conversely, if the epi-mono factorisation of $f$ coincides with the active-inert factorisation then the first map in \eqref{diag:epimono} is an active surjection.
	That is, the image of $f$ is an interval in $[n]$ which forces $f$ to have the desired property.
\end{proof}


%% file: sec_monad.tex
\section{Fat Delta via monads with arities}\label{sec:frelsem}
In this section, we use the theory of monads with arities extensively.
We begin by constructing the semicategory monad $\free$ and the free relative semicategory monad $\rfree$ on the categories of graphs and relative graphs, respectively.
We show that $\rfree$ is strongly cartesian by using the result that $\free$ is also strongly cartesian.
Following this, we develop the category of arities for $\rfree$ and make precise its connection to $\fDel$.

\subsection{The free semicategory monad}\label{sec:fscm}
We denote by $\Graph$ the category of \emph{directed graphs} (also known as quivers) and graph morphisms.
More precisely, $\Graph$ is the presheaf category over the truncated globe category $\mathbb{G}_{1}$, see \cite[Chapter 5]{Leinster2004}.
Since we have a fully faithful embedding $i\:\mathbb{G}_{1} \hookrightarrow \SemiCat$, we get a forgetful functor $\sForget \: \SemiCat \to \Graph$ by mapping a semicategory $\C$ to the graph $\SemiCat(i,\C)$.
Taking the left Kan extension along the Yoneda embedding we get the free-forgetful adjunction \autocite[Proposition 3.2.2]{Loregian2021}:
\begin{equation}\label{diag:fuadj}
	\adjunction{\sFr}{\Graph}{\SemiCat}{\sForget}.
\end{equation}
An explicit description of the functor $\sFr$ is given by the \emph{path semicategory} of a graph.
That is, for any graph $X$, the semicategory $\sFr X$ is the semicategory with the same collection of objects $X_0$ as $X$ and a morphism between two vertices $u$ and $v$ is a path $[n] \to X$, where $[n]$ denotes the linearly ordered graph of length $n$, starting at $u$ and ending at $v$, for $n \geq 1$.
We denote by $\Path_n X$ the hom-set $\Graph([n],X)$ of paths of length $n$ in $X$, and by $\Path X$ the set $\bigsqcup_{n\geq 1} \Path_n X$ of paths of any length.
The composition of two paths $\pathp$, $\pathq \in \Path X$ is given by the \emph{concatenation} $\pathp\vpro\pathq$.

\begin{rmk}\label{rmk:concat}
	Suppose that $\pathp$ and $\pathq$ are of length $n$ and $m$, respectively. Concatenation of paths in $\Graph$ is constructed by the pushout
	\begin{equation*}
	\begin{tikzcd}
		{[0]}
			\ar[r, "n"]
			\ar[d, "0"']
			\ar[dr, phantom, very near end, "{\ulcorner}"]
		& \pathp
			\ar[d] \\
		\pathq
			\ar[r]
		& \pathp\vpro\pathq
	\end{tikzcd}
	\end{equation*}
	Note that concatenation is denoted by the same symbol as the $\vpro$-product for $\Delta$ in \cref{sec:activeinert}.
	This is not a coincidence as they are related: if $\Fr$ is the free category functor, then $\Fr(\pathp \vpro \pathq)={\Fr\pathp} \vpro {\Fr\pathq}$. However, the context will make it clear which one is being used.
\end{rmk}

This adjunction gives rise to the \emph{free semicategory monad} $\free = (\sForget\sFr,\sForget\epsilon_{\sFr},\eta)$ on $\Graph$, where $\eta$ and $\epsilon$ are the unit and counit of the adjunction.
We will denote by $\nu$ and $\mu$ the unit and multiplication of the monad $\free$.

\begin{rmk}\label{rmk:free}
	Let $G$ be in $\Graph$, the graph $\free G$ is the underlying graph of the path semicategory. It has $G_0$ as its set of vertices and $(\free G)_1= \Path G$.
	In particular, if $\gamma\:G \to G'$ is a morphism in $\Graph$, the morphism $\free \gamma$ turns a path $\pathp \in \Path_n G$ into a path $\gamma \pathp\:[n] \to G'$ by post-composing with $\gamma$.
	The unit $\nu_G\:G \to \free G$ sends an edge $g$ in $G$ to the corresponding path $\freem{g}\:[1] \to G$.
	The edges in $\free\free G$ are paths of paths $\pathp\:[n] \to \free G$. Thus, using the counit of the adjunction \eqref{diag:fuadj}, the multiplication $\mu_G\:\free\free G \to \free G$ can be represented as mapping $\pathp$ to the concatenation $\pathp_1 \vpro\cdots\vpro \pathp_n \in \Path \free G$.
\end{rmk}

\begin{prop}\label{prop:freeCart}
	The monad $\free\:\Graph \to \Graph$ is cartesian.
\end{prop}
\begin{proof}
	The category $\Graph$ has pullbacks, as it is a presheaf category. Let us first show that $\free$ preserves pullbacks. For a cospan $G \to K \leftarrow H$ in $\Graph$ we get by definition that
	$ \free(G \times_K H)_0 = (G \times_K H)_0 = G_0 \times_{K_0} H_0 = \free(G)_0 \times_{\free(K)_0} \free(H)_0$,
	and that $\free(G \times_K H)_1 = \Path (G \times_K H)$. The universal property of pullbacks and a direct calculation gives us the natural isomorphisms
	\begin{equation*}
		\Path (G \times_K H) \cong \bigsqcup_{n\geq 1}\Path_n G \times_{\Path_n K} \Path_n H
	\cong \Path G\times_{\Path K} \Path H.
	\end{equation*}
	Thus, $\free(G \times_K H)_1=\free(G)_1 \times_{\free(K)_1} \free(H)_1$ and the functor $\free\:\Graph \to \Graph$ preserves pullbacks.
	Using \cref{lem:charCartesianMonad}, it remains to show that the naturality squares
	\begin{equation*}
	\begin{tikzcd}
		G
		\ar[r,"\nu_G"]
		\ar[d,"!_G"'] &
		\free G
		\ar[d,"\free !_G"] \\
		\tOb{\Graph}
		\ar[r,"\nu_{\tOb{\Graph}}"'] &
		\free \tOb{\Graph}
	\end{tikzcd}
	\quad
	\text{and}
	\quad
	\begin{tikzcd}
		\free \free G
		\ar[r,"\mu_G"]
		\ar[d, "\free \free !_G"'] &
		\free G
		\ar[d,"\free !_G"] \\
		\free \free \tOb{\Graph}
		\ar[r, "\mu_{\tOb{\Graph}}"'] &
		\free \tOb{\Graph}
	\end{tikzcd}
	\end{equation*}
	for the unit and for the multiplication are cartesian, for any graph $G$. The first square is cartesian because $G$ is exactly the fibre of $\free !_G$ over the path corresponding to the loop in $\tOb{\Graph}$.
	The pullback of the cospan of the right square, $\free \free \tOb{\Graph} \times_{\free\tOb{\Graph}} \free G$, has $G_0$ as its set of objects and $\Path \free \tOb{\Graph}\times_{\Path \tOb{\Graph}} \Path G$ as its set of arrows.
	Note that, since $\free \tOb{}$ can be viewed as the set of strictly positive natural numbers, an edge $k \in \free\free\tOb{}$ corresponds to a non-empty finite sequence of strictly positive natural numbers $(k_1, \ldots, k_n)$.
	Thus, by \cref{rmk:free}, $\mu_{\tOb{}}$ maps $k$ to the sum $k_1+\cdots+k_n$, and $\free !_G$ maps a path $\pathp\in \Path G$ to its length~$l_{\pathp}$.
	So, the set of arrows corresponds to pairs $(k,\pathp) \in \Path\free\tOb{\Graph} \times \Path G$ such that $l_{\pathp} = k_1+\cdots+k_n$.
	This can be rewritten as the set of paths in $G$ equipped with a partition, which then just corresponds to the set of paths in $\free G$.
	In other words, $\free \free \tOb{\Graph} \times_{\free \tOb{\Graph}} \free G \cong \free\free G$ and the square is cartesian.
\end{proof}

Denote by $\genm{\bfm{n}}\:[1] \to \free [n]$ the map that picks the maximal path in $[n]$.
More precisely, if we write $01$ for the only edge of $[1]$ then $\genm{\bfm{n}}(01)=\id_{[n]}$.

\begin{lem}\label{lem:maxpathGeneric}
	The maps $\genm{\bfm{n}}\:[1] \vbto \free [n]$ are $\free$-generic, for all $n \geq 1$.
\end{lem}
\begin{proof}
	Consider $\alpha$, $\beta$, and $\gamma$ making the following square commute:
	\begin{center}
	\begin{tikzcd}
		{[1]}
		\ar[r,"\alpha"]
		\ar[d,"\genm{\bfm{n}}"', mid vert] &
		\free X
		\ar[d,"\free \gamma"] \\
		\free [n]
		\ar[r, "\free \beta"'] &
		\free Y
	\end{tikzcd}
	\end{center}
	The map $[1] \xrightarrow{\alpha} \free X$ picks a path $\langle\alpha\rangle$ in $X$. The commutativity of this square ensures that $\gamma \langle\alpha\rangle=\beta$.
	Moreover, since $\beta$ is of length $n$, the path $\langle\alpha\rangle\:[n] \to X$ is of length $n$ as well.
	In particular, the fact that it is of length $n$ induces a factorisation of $\alpha$ as $\free\langle\alpha\rangle\,\genm{\bfm{n}}$.
	Consequently, we can choose $\langle\alpha\rangle$ as an $\free$-filler of the square.
	It is unique because, for any other path $p\:[n] \to X$, the commutativity of the top triangle enforces $p=\langle\alpha\rangle$.
\end{proof}

The following result is the analogue of \cite[Example 2.5]{Weber2007} for the case of semicategories instead of categories.

\begin{prop}\label{prop:freeStrgCart}
	The functor $\free\:\Graph \to \Graph$ is a local right adjoint.
\end{prop}
\begin{proof}
	The category of graphs $\Graph$ is a presheaf category, so by \cite[Proposition 2.10]{Weber2007} (or its generalisation \cref{prop:LRAdenseGen}) it suffices to show that the maps $[0]\to \free\tOb{\Graph}$ and $[1]\to \free\tOb{\Graph}$ have $\free$-generic factorisations.
	Since $[0]=\free [0]$ the first map has a trivial $\free$-generic factorisation
	\begin{equation*}
		[0] \xlongequal{} \free [0] \to \free \tOb{\Graph}.
	\end{equation*}
	The second map $\bfm{n}\:[1] \to \free\tOb{\Graph}$ picks a path in $\tOb{\Graph}$, that is a graph morphism $\freem{\bfm{n}}\:[n] \to \tOb{\Graph}$, where $n$ is the length of the path. The morphism $\bfm{n}$ factors as
	\begin{equation*}
		[1] \xvblongto{\genm{\bfm{n}}} \free[n] \xlongrightarrow{\free \freem{\bfm{n}}} \free\tOb{\Graph}.
	\end{equation*}
	This is an $\free$-generic factorisation by \cref{lem:maxpathGeneric}.
\end{proof}

In particular, the monad $\free$ is strongly cartesian.
Using \cref{thm:nerve,thm:scmag}, we can deduce the nerve theorem for $\Delta_+$ and $\SemiCat$.

\subsection{The free relative semicategory monad}\label{sec:frscm}
The category of \emph{relative graphs} $\RelGraph$ consists of directed graphs $\carrier{X}$ equipped with a wide subgraph, represented by a bijective-on-objects monomorphism $i_X\:\marking{X}\hookrightarrow \carrier{X}$.
 We will often identify the map $i_X$ with the relative graph $X = (\carrier{X},\marking{X})$.
In other words, the category $\RelGraph$ corresponds to the full subcategory $\Arr{\Graph}_{\Mbo} \hookrightarrow \Arr{\Graph}$, where $\Mbo$ denotes the class of bijective-on-objects monomorphisms.

\begin{rmk}\label{rmk:sepgra}
	The category $\RelGraph$ can equivalently be described as the category of separated presheaves (see \cite[Section C.2.2]{Elephant2} for separated presheaves) over the category $\marking{\mathbb{G}}_1$.
	This category consists of $\mathbb{G}_1$ but with one more object $[\marking{1}]$ and a morphism $[1] \to [\marking{1}]$, equipped with the coverage whose only non-trivial covering family is $\left\{ [1] \to [\marking{1}] \right\}$.
\end{rmk}

\begin{rmk}\label{rmk:adjrg}
	Similarly to \cref{rmk:adjrsc}, we have a quadruple adjunction
	\begin{equation*}
		(\blank)^{\flat} \dashv \carrier{(\blank)}  \dashv (\blank)^{\sharp} \dashv \marking{(\blank)}
	\end{equation*}
	where the forgetful functors $\carrier{(\blank)},\marking{(\blank)}\:\RelGraph \to \Graph$ have left adjoints obtained by Yoneda extensions, and also denoted by $(\blank)^{\sharp}$ and $(\blank)^{\flat}$.
	They respectively mark all edges or none of them.
\end{rmk}

As a consequence of \cref{prop:freeCart}, the functor $\free$ preserves monomorphisms and since $\free$ is the identity on the set of vertices it also preserves the class $\Mbo$.
Thus, as explained in \cref{sec:monar}, the monad $\free$ induces a monad on $\RelGraph$. However, this monad can also be described as the monad coming from an adjunction induced by the adjunction $\sFr \dashv \sForget$ from~\eqref{diag:fuadj}.

\begin{prop}
	The adjunction $\sFr \dashv \sForget$ can be lifted to a free-forgetful adjunction
	\begin{equation}
		\adjunction{\rFr}{\RelGraph}{\RelSemiCat}{\rForget}.
	\end{equation}
\end{prop}
\begin{proof}
	Define $\rFr\:\RelGraph \to \RelSemiCat$ by mapping a relative graph $X$ to $\sFr(i_X)$, and $\rForget\:\RelSemiCat \to \RelGraph$ analogously. By adjointness $\sFr \dashv \sForget$, we have a one-to-one correspondence between the squares
	\begin{equation*}
	\begin{tikzcd}
		\sFr\marking{X} \ar[r] \ar[d, hook']
		& \marking{Y} \ar[d, hook] \\
		\sFr\carrier{X} \ar[r]
		& \carrier{Y}
	\end{tikzcd}
	\quad \text{and} \quad
	\begin{tikzcd}
		\marking{X} \ar[r] \ar[d, hook']
		& \sForget\marking{Y} \ar[d, hook] \\
		\carrier{X} \ar[r]
		& \sForget\carrier{Y}
	\end{tikzcd}
	\end{equation*}
	In particular, this gives rise to a natural bijection between the hom-sets $\RelSemiCat(\rFr X, Y)$ and $\RelGraph(X,\rForget Y)$, thus $\rFr \dashv \rForget$.
\end{proof}

We will call the monad induced by the adjunction $\rFr \dashv \rForget$ the \emph{free relative semicategory monad} and denote it by $\rfree$.

\begin{rmk}\label{rmk:rfree}
	More precisely, the functor $\rfree\: \RelGraph \to \RelGraph$ maps a relative graph $X$ to $(\free \carrier{X}, \free\marking{X})$ where $\free \carrier{X}$ is the graph of the free semicategory on the carrier of $X$, and $\free\marking{X}$ is the graph of the free semicategory of $\marking{X}$.
\end{rmk}

We can also describe the graph $\rfree X$ as the graph of the relative path semicategory $\rPath X$.
The objects are the vertices of $X$ and the morphisms are the \emph{(relative) paths} in $X$.
That is, elements of the hom-set $\RelGraph(\alpha,X)$, where $\alpha$ is a relative linearly ordered graph of length $n \geq 1$.
The composition of two relative paths is given by \emph{(relative) concatenation} $\pathp\vpro\pathq$, which is constructed as in \cref{rmk:concat} but within $\RelGraph$.
Similarly to \cref{rmk:free}, the unit $\nu\:X \to \rfree X$ sends the edges to the corresponding paths of length one and, in addition, respects the marking.
The multiplication $\mu\:\rfree\rfree X \to \rfree X$ is given by concatenation.

\begin{prop}
	The adjunction $\rFr\dashv\rForget$ is monadic. In particular, we have an equivalence of categories $\EM{\rfree} \simeq \RelSemiCat$.
\end{prop}
We follow the proof given in the second half of the demonstration of~\cite[Corollary~3.4.7]{TTT1985} but adapt it to the relative semicategories and relative graphs.
\begin{proof}
	It suffices to check that $\rForget$ satisfies the hypotheses of Beck's monadicity theorem, as formulated in~\cite[Theorem~4.4.4]{Borceux1994}.
	The third condition of the monadicity theorem is met.
	In fact, let $f,g\:\C \to C'$ be in $\RelSemiCat$ and suppose the coequaliser $h \: \rForget\C' \to G$ for the pair $(\rForget(f),\rForget(g))$ is absolute.
	Using the counit $\epsilon$ of the adjunction $\rFr \dashv \rForget$, we get the following commutative diagram:
	\begin{equation*}
	\begin{tikzcd}
		\rfree\rForget \C
			\ar[r, "\rfree\rForget(f)", shift left]
			\ar[r, "\rfree\rForget(g)"', shift right]
			\ar[d, "\epsilon_{\rForget \C}"'] &
		\rfree\rForget \C'
			\ar[r, "\rfree h"]
			\ar[d, "\epsilon_{\rForget \C'}"] &
		\rfree G
			\ar[d, "\gamma", dotted]\\
		\rForget \C
			\ar[r, "\rForget(f)", shift left]
			\ar[r, "\rForget(g)"', shift right] &
		\rForget \C'
			\ar[r, "h"'] &
		G
	\end{tikzcd}
	\end{equation*}
	The canonical map $\gamma\:\rfree G \to G$ coming from the fact that $\rfree h$ is the coequaliser gives the structure of a relative semicategory to $G$.
	In addition, the relative semicategory associated to $\gamma$ has the universal property of the coequaliser of $(f,g)$, and this coequaliser is preserved by $\rForget$.

    We now check that $\rForget$ reflects isomorphisms.
	Let $f\:\C \to \C'$ be a morphism of relative semicategories and suppose that $\rForget(f)$ is an isomorphism.
	The inverse $g\:\rForget \C' \to\rForget \C$ lifts to a functor of relative semicategories since, for an edge $x_0\, x_1$ of $\rForget \C'$, we have
	\begin{equation*}
		g(x_0\,x_1)=g((\rForget(f)g (x_0))\,(\rForget(f)\,g (x_1))) = g(\rForget(f)(g (x_0)\, g (x_1)))
	\end{equation*}
	by using the functoriality of $f$.
	Thus, $g(x_0\, x_1)=g (x_0)\, g (x_1)$.
\end{proof}

In view of showing that $\rfree$ is strongly cartesian, we will first show that it is a local right adjoint using \cref{prop:LRAdenseGen}.
Recall from \cref{rmk:sepgra} that $\RelGraph$ is a category of separated presheaves.
In particular, it has the full subcategory on $[0]$, $[1]^{\flat}$ and $[1]^{\sharp}$ as a dense generator, which can be identified with $(\marking{\mathbb{G}}_1)^{\op}$.
We show that $\rfree$ is a local right adjoint by taking $\rfree$-generic factorisations of the maps $[0] \to \rfree \tOb{\RelGraph}$, $[1]^{\flat} \to \rfree \tOb{\RelGraph}$ and $[1]^{\sharp} \to \rfree \tOb{\RelGraph}$.
Here, the notation $(\blank)^{\flat}$ and $(\blank)^{\sharp}$ was introduced in \cref{rmk:adjrg}, and $\rfree\tOb{\RelGraph}$ corresponds to the relative graph with one object and a marked loop for each natural number $n\geq 1$.
Thus, for $\sigma \in \{\flat,\sharp\}$, a morphism $\bfm{n}^{\sigma}\:[1]^{\sigma} \to \rfree\tOb{\RelGraph}$ selects a single loop of length $n$, which corresponds to a path $\freem{\bfm{n}}^{\sigma}\:[n]^{\sigma} \to \tOb{\RelGraph}$.
Let us consider the maximal path in $[n]^{\sigma}$, namely $\id_{[n]^{\sigma}}$, and the morphism $\genm{\bfm{n}}^{\sigma}\: [1]^{\sigma} \xlongrightarrow{} \rfree [n]^{\sigma}$ that picks it out.
We also denote by $\{n\}$ the discrete graph with $n+1$ vertices.

\begin{lem}\label{lem:maxrelpathGeneric}
	The maps $\genm{\bfm{n}}^{\sigma}\: [1]^{\sigma} \vblongto \rfree [n]^{\sigma}$ are $\rfree$-generic.
\end{lem}
\begin{proof}
	Let $\alpha$, $\beta$ and $\gamma$ be morphisms making the diagram
	\begin{center}
		\begin{tikzcd}
			{[1]^{\sigma}}
			\ar[r, "\alpha"]
			\ar[d, "\genm{\bfm{n}}^{\sigma}"', mid vert] &
			\rfree X
			\ar[d, "\rfree\gamma"] \\
			{\rfree [n]^{\sigma}}
			\ar[r, "\rfree\beta"'] &
			\rfree Y
		\end{tikzcd}
	\end{center}
	commute.
	Recall that a square like this in $\RelGraph$ can be represented by a cube with the carrier graphs on the front square and the wide subgraphs of marked edges on the back square.
	\begin{center}
	\begin{tikzcd}[row sep=0.8cm, column sep=1cm]
		&
		{\marking{([1]^\sigma)}}
		\ar[dl, "i_{[1]^{\sigma}}" description, hook']
		\ar[rr, "\marking{\alpha}" description]
		\ar[dd, "\marking{(\genm{\bfm{n}}^{\sigma})}" description, near end]
		& &
		\free\marking{X}
		\ar[dl, "i_{\free X}" description, hook']
		\ar[dd, "\marking{\free\gamma}" description] \\
		{[1]}
		\ar[rr, "\carrier{\alpha}"description, near end, crossing over]
		\ar[dd, "\carrier{(\genm{\bfm{n}}^{\sigma})} \,=\, \genm{\bfm{n}}"', mid vert]
		& &
		\free\carrier{X} \\
		&
		\free{\marking{([n]^\sigma)}}
		\ar[dl, "i_{\rfree{[n]^\sigma}}" description, hook']
		\ar[rr, "\marking{\free \beta}" description, near start]
		& &
		\free\marking{Y}
		\ar[dl, "i_{\free Y}" description, hook'] \\
		\free{[n]}
		\ar[rr, "\carrier{\free \beta}" description]
		& &
		\free\carrier{Y}
		\ar[from=uu,"\carrier{\free \gamma}"description, crossing over, near end]
	\end{tikzcd}
	\end{center}
	For $\sigma \in \{\sharp,\flat\}$, the carrier map $\carrier{(\genm{\bfm{n}}^{\sigma})}$ is the morphism $\genm{\bfm{n}}$, which is $\free$-generic by \cref{lem:maxpathGeneric}.
	In particular, the front square has an $\free$-filler
	$\carrier{\delta}\:[n] \to \carrier{X}$.\\
	For $\sigma = \sharp$ the morphism $\marking{(\genm{\bfm{n}}^{\sharp})}$ is also equal to $\genm{\bfm{n}}$, hence the maps $i_{[1]^{\sigma}}$ and $i_{\rfree{[n]^\sigma}}$ are identities and the back square has an $\free$-filler $\marking{\delta}\:[n] \to \marking{X}$.
	By uniqueness of $\carrier{\delta}$ we have the equality $\carrier{\delta} = i_{\free X} \marking{\delta}$, so we get a map $\delta\:[n]^{\sharp} \to X$.
	The relations $\genm{\bfm{n}}^{\sharp} \rfree\delta = \alpha$ and $\gamma \delta = \beta$, and the uniqueness making $\delta$ a $\rfree$-filler hold trivially since $\carrier{\delta}$ and $\marking{\delta}$ are $\free$-fillers.\\
	For $\sigma = \flat$ no edge is marked, therefore the subgraphs $\marking{([1]^{\flat})}=\{1\}$ and $\marking{(\rfree [n]^{\flat})}=\{n\}$ are discrete.
	Thus, the morphism $\marking{(\genm{\bfm{n}}^{\flat})}\:\{1\} \to \{n\}$ simply maps $0$ to $0$ and $1$ to $n$, and $\carrier{\delta}$ induces a morphism $\marking{\delta}\: \{n\} \to \marking{X}$.
	In fact, the composite $\carrier{\delta} i_{\rfree [n]^{\flat}}$ picks out the sequence of vertices determined by $\carrier{\delta}$, and hence can be restricted to $\marking{X}$ since it is a wide subgraph of $X$.
	In particular, we get a map $\delta\:[n] \to X$. Since $\marking{\delta}$ is totally determined by $\carrier{\delta}$, we get $\gamma \delta= \beta$ and $\rfree \delta \, \genm{\bfm{n}}^{\flat}= \alpha$. So, $\delta$ is an $\rfree$-filler.
\end{proof}

\begin{prop}\label{prop:relpathGenFact}
	The maps $\bfm{0}\:[0] \to \rfree \tOb{\RelGraph}$ and $\bfm{n}^{\sigma}\:[1]^{\sigma} \to \rfree \tOb{\RelGraph}$
	have $\rfree$\nobreakdash-generic factorisations given by
	\begin{equation*}
		[0] \xvblongto{\id_{[0]}} \rfree[0]=[0] \xlongrightarrow{\bfm{0}} \rfree \tOb{\RelGraph} \quad \text{and} \quad
		[1]^{\sigma} \xvblongto{\genm{\bfm{n}}^{\sigma}} \rfree [n]^{\sigma} \xlongrightarrow{\rfree \freem{\bfm{n}}^{\sigma}} \rfree\tOb{\RelGraph},
	\end{equation*}
	for $n\geq 1$ and $\sigma \in \{\flat,\sharp\}$.
\end{prop}
\begin{proof}
	The factorisation for $\bfm{0}$ is trivial and for $\rfree \freem{\bfm{n}}^{\sigma} \, \genm{\bfm{n}}^{\sigma}=\bfm{n}^{\sigma}$ it suffices to verify its action on the only edge $01^{\sigma} \in [1]^{\sigma}$.
	We have $\rfree \freem{\bfm{n}}^{\sigma} \, \genm{\bfm{n}}^{\sigma}(01^{\sigma})=\rfree\freem{\bfm{n}}^{\sigma}(\id_{[n]^{\sigma}})=\freem{\bfm{n}}^{\sigma} \, \id_{[n]^{\sigma}}$, so $\rfree\freem{\bfm{n}}^{\sigma} \, \genm{\bfm{n}}^{\sigma} = \bfm{n}^{\sigma}$.
	The result now follows from \cref{lem:maxrelpathGeneric}.
\end{proof}

\begin{prop}\label{prop:rfreeStrgCart}
	The monad $\rfree\:\RelGraph \to \RelGraph$ is strongly cartesian.
\end{prop}
\begin{proof}
	The monad $\rfree$ is induced by $\free$, which is strongly cartesian by \cref{prop:freeCart,prop:freeStrgCart}.
	Moreover, $\RelGraph$ is cartesian since it is a separated presheaf category (see \cref{rmk:sepgra}), and in particular a quasitopos admitting all finite limits, so $\rfree$ is cartesian by \cref{prop:MonadClass}.
	The result follows from \cref{prop:LRAdenseGen} applied to \cref{prop:relpathGenFact}.
\end{proof}

Thus, we can provide the free relative semicategory monad $\rfree$ with arities by fixing a dense generator of $\RelGraph$.
The category of arities of this monad is the subject of the next section.

\subsection{The arities of the free relative semicategory monad}\label{sec:arities}
In this section, we construct the category of canonical arities $\Ar{\Alt}{\rfree}$ for the free relative semicategory monad $\rfree$.
Before doing this, we introduce a particular dense generator $\Alt$ and give some intuition for our choice.

\subsubsection{Intuition for $\Alt$}
Thanks to \cref{thm:scmag}, we know that $\rfree$ has (canonical) arities, which depend on the chosen dense generator.
For both categories and semicategories, the smallest dense generator of $\Graph$ is taken for the construction of the canonical arities. For $\fDel$, we need to consider a non-trivial dense generator.
Indeed, if we take the smallest dense generator of $\RelGraph$ (the full subcategory generated by $[0]$ and $[1]^{\sigma}$, with $\sigma\in\{\flat,\sharp\}$) then, by \cref{prop:relpathGenFact}, the canonical arities correspond to the full subcategory of linearly ordered graphs $[0]$ and $[n]^{\sigma}$, for $n > 0$, that are either fully marked or unmarked.
This only captures a subcategory of $\fDel$, therefore one needs to refine the dense generator to include the partially marked objects.
The dense subcategory which we consider is the full subcategory $\Alt$ of alternatingly marked linearly ordered graphs, and we motivate this choice from two perspectives.

\paragraph*{}
The first perspective comes from the $\rfree$-generic factorisations and the description of $\fDel$ as (finite non-empty) relative semiordinals.
\cref{prop:relpathGenFact} shows that maps $[1]^\sigma \to \rfree\tOb{}$ factor generically through $\rfree[n]^\sigma$ where $[n]^\sigma$ is a fully (un)marked string of some length $n$.
Consequently, any relative graphs with consecutive edges of the same marking can be seen as being generated from $[1]^{\sigma}$ by $\rfree$.

An arbitrary relative semiordinal $\eta$ is a sequence of fully marked and unmarked strings of any length, where the marking of the strings alternates.
Each of these strings can be seen as being generated from $[1]^{\sigma}$ by $\rfree$.
Therefore, one can expect $\rfree$ to build $\eta$ from a sequence of $[1]^{\sigma}$, where $\sigma$ corresponds to the marking of the string.
This reasoning leads to the choice of the full subcategory of those linearly ordered graphs with no two adjacent edges of the same marking, which we denote by $\Alt$.

\paragraph*{}
With an eye on a future study of the coherences of weak identity structures in higher categories, we would like to interpret the marked morphisms in relative semicategories as weak identities and morphisms between marked strings as coherences.
The simplest and minimal identity coherences for categories are the left and right neutrality, which can be represented by morphisms between $[1]^{\flat}$ and marked linearly ordered graphs as in \eqref{diag:lrneut}:
\begin{equation}\label{diag:lrcoh}
\begin{tikzcd}[row sep=0.5, column sep=35]
	\overset{\rotatebox{-90}{$\langle$}}{\bullet}
		\ar[dd]
		\ar[dddd, "\scriptstyle{\id f}"', phantom, bend right=60]
	&
	& \overset{\rotatebox{-90}{$\langle$}}{\bullet}
		\ar[dddd, "\scriptstyle{f\id}"', phantom, bend left=60]
		\ar[dd, red] \\
	&\overset{\rotatebox{-90}{$\langle$}}{\bullet}
		\ar[dd,"\hspace{5pt}f"]
		\ar[ul, bend right, mapsto, dashed]
		\ar[ur, bend left, mapsto, dashed]
	& \\
	\bullet
		\ar[dd, red]
	&
	& \bullet
		\ar[dd] \\
	& \underset{\rotatebox{90}{$\langle$}}{\bullet}
		\ar[dl, bend left, mapsto, dashed]
		\ar[dr, bend right, mapsto, dashed]
	& \\
	\underset{\rotatebox{90}{$\langle$}}{\bullet}
	&
	& \underset{\rotatebox{90}{$\langle$}}{\bullet} \\
\end{tikzcd}
\end{equation}
Therefore, to extend \eqref{diag:lrcoh} to higher coherences for the identity structure, we need to increase the length of the strings by adding edges.
To maintain simplicity and minimality of the coherences, we only add edges in an alternating marking pattern.
The more complex ones are then constructed as the canonical arities using $\rfree$.
For instance, for a string of length $4$, we either have:
\begin{align}
	&\langle\bullet \mlgar \bullet \longrightarrow \bullet \mlgar \bullet \longrightarrow \bullet\rangle\label{diag:altex0}, \\
	\text{or} \quad
	&\langle\bullet \longrightarrow \bullet \mlgar \bullet \longrightarrow \bullet \mlgar \bullet\rangle\label{diag:altex1}
\end{align}
This type of string corresponds to alternatingly marked linear graphs, that is, we recover the full subcategory $\Alt$ of $\RelGraph$ from the previous perspective.
Most importantly, this category is dense in $\RelGraph$ since it contains the minimal dense generator.

\begin{rmk}\label{rmk:AltObj}
	As should be clear from \eqref{diag:altex0}-\eqref{diag:altex1}, there are only two possible configurations for an alternatingly marked linear graph of length $n\geq 1$: it either starts with a marked edge or with an unmarked edge.
	The marking of the final edge $e_n$ is determined by the parity of $n$.
\end{rmk}

\subsubsection{Construction of the arities}\label{sec:coar}
 We will denote by $\altl{\alto{n}}$ an object of $\Alt$ of length $n$ and starting with a marked edge as in \eqref{diag:altex0}, and write $\altg{\alto{n}}$ if it starts with an unmarked edge as in \eqref{diag:altex1}.
 In general, we will write $\alto{n}^{\epsilon}$, with $\epsilon\in\{0,1\}$.
Moreover, we write $\sigma^0=(\sharp,\flat,\ldots)$ and $\sigma^1=(\flat,\sharp, \ldots)$ for the (finite) marking sequences corresponding to $\altl{\alto{n}}$ and~$\altg{\alto{n}}$, and we leave the length implicit.
The category of canonical arities of $\rfree$, relative to $\Alt$, is defined by taking the $\rfree$-generic factorisations of all maps $\pathp\:\alto{n}^{\epsilon} \xlongrightarrow{} \rfree\tOb{\RelGraph}$, for $\alto{n}^{\epsilon}$ in $\Alt$.
From \cref{rmk:AltObj}, the objects $\alto{n}^{\epsilon}$ can be decomposed as follows:
\begin{equation}\label{eq:aldec}
	\alto{n}^{\epsilon} = [1]^{\sigma^{\epsilon}_1} \vpro\cdots\vpro [1]^{\sigma^{\epsilon}_n}.
\end{equation}
A map $\pathp\:\alto{n}^{\epsilon} \xlongrightarrow{} \rfree\tOb{\RelGraph}$ can then be described by its action on edges.
That is, by the maps $\pathp_i\: [1]^{\sigma^{\epsilon}_i}  \xlongrightarrow{\tilde{e_i}} \alto{n}^{\epsilon} \xlongrightarrow{\pathp} \rfree\tOb{\RelGraph}$, for $i=1,\ldots,n$.
These morphisms have $\rfree$-generic factorisations given by \cref{prop:relpathGenFact}:
\begin{equation*}
	[1]^{\sigma^{\epsilon}_i} \xvblongto{\genm{\pathp}_i} \rfree [m_i]^{\sigma^{\epsilon}_i} \xlongrightarrow{\rfree\freem{\pathp}_i} \rfree\tOb{\RelGraph}.
\end{equation*}
Altogether, the $\rfree$-generic morphisms $\genm{\pathp}_i$ form the natural transformation denoted by $[1]^{\sigma^{\epsilon}_{(\blank)}} \vblongto \rfree [m_{(\blank)}]^{\sigma^{\epsilon}_{(\blank)}}$ between diagrams over the zigzag category corresponding to the decomposition \eqref{eq:aldec} of $\alto{n}^{\epsilon}$.
In particular, we have a canonical morphism $\genm{\pathp}\: \alto{n}^{\epsilon} \longrightarrow{} \rfree\eta_{\pathp}$, where $\eta_{\pathp}$ corresponds to the linear relative graph $[m_1]^{\sigma^{\epsilon}_1} \vpro\cdots\vpro [m_n]^{\sigma^{\epsilon}_n}$.
Using \cref{lem:ColimPresGen}, the following result is immediate.

\begin{prop}\label{prop:etaGeneric}
	The map $\genm{\pathp}\: \alto{n}^{\epsilon} \vblongto \rfree\eta_{\pathp}$ is $\rfree$-generic.\qed
\end{prop}

More precisely, $\genm{\pathp}$ maps an edge $e_i$ in $\alto{n}^{\epsilon}$ to the canonical morphism $j_i\:[m_i]^{\sigma^{\epsilon}_i} \to \eta_{\pathp}$, which corresponds to the relative path of length $m_i$ that starts at $m_1 +\cdots+ m_{i-1}$.
The image of $\pathp\:\alto{n}^{\epsilon} \to \rfree\tOb{\RelGraph}$ can hence be described as the concatenation of all the paths $\freem{\pathp}_i\:[m_i]^{\sigma^{\epsilon}_i} \to \tOb{}$.
This concatenation is denoted by $\freem{\pathp}\:\eta_{\pathp} \to \tOb{\RelGraph}$.

\begin{cor}
	A map $\pathp\:\alto{n}^{\epsilon} \xlongrightarrow{} \rfree\tOb{\RelGraph}$ has an $\rfree$-generic factorisation given by
	\begin{equation*}
		\alto{n}^{\epsilon} \xvblongto{\genm{\pathp}} \rfree\eta_{\pathp} \xlongrightarrow{\rfree\freem{\pathp}}  \rfree\tOb{\RelGraph}.
	\end{equation*}
\end{cor}
\begin{proof}
	We have that $\rfree \freem{\pathp} \, \genm{\pathp}(e_i) = \freem{\pathp} j_i = \pathp(e_i)$, where $j_i\:[m_i]^{\sigma^{\epsilon}_i} \to \eta_{\pathp}$.
	The result then follows from \cref{prop:etaGeneric}.
\end{proof}

The $\rfree$-generic factorisations of the corollary are essentially unique by \cite[Lemma 5.7]{Weber2004}.

It follows from the decomposition $\eta_{\pathp}=[m_1]^{\sigma^{\epsilon}_1} \vpro\cdots\vpro [m_n]^{\sigma^{\epsilon}_n}$ and the definition that the category of canonical arities $\Ar{\Alt}{\rfree}$ is the full subcategory of partially marked linear graphs, with relative graph morphisms between them.

\begin{rmk}\label{rmk:GraphMorphDistPres}
	Let $g\:[n] \to [m]$ be a graph morphism of linear graphs. Writing $\src$ and $\trg$ for the source and target maps, and $g_0$ and $g_1$ for the object and morphism parts of $g$, we have by definition that $\trg g_1 = g_0 \trg$ and $\src g_1 = g_0 \src$.
	Let $e_i$ denote the $i^{\text{th}}$ edge in $[n]$, then $g_0(i+1)=g_0 \trg(e_i)= \trg g_1(e_i)$.
	Since $g_1(e_i)$ is an edge of length $1$ from $g_0(i)$ we get that $g_0(i+1)= g_0(i) + 1$.
	In particular, morphisms of linear graphs are distance preserving and monomorphisms.
	This property extends to the relative case, as the graph morphisms are the same with the additional requirement that they preserve marked edges.
\end{rmk}

As $\fDel_0$ is defined (in \cref{sec:activeinert}) as the (wide) subcategory of distance-preserving morphisms, the above \cref{rmk:GraphMorphDistPres} suggests a close connection to the category of canonical arities $\Ar{\Alt}{\rfree}$ and $\fDel_0$.
Indeed, we will establish in \cref{thm:arfDel} an isomorphism of categories between $\Ar{\Alt}{\rfree}$ and $\fDel_0$. To achieve this, we construct inverse functors.
We start with $\phi\:\Ar{\Alt}{\rfree} \to \fDel_0$.

\begin{constr}\label{constr:phi}
	Let $\alpha$ be an object in $\Ar{\Alt}{\rfree}$, in particular $\alpha$ corresponds to a wide monomorphism $i_{\alpha}\:\marking{\alpha} \hookrightarrow [n]$, where $\marking{\alpha}$ is an alternating sequence of linear graphs ${[n_1],\ldots,[n_k]}$ and discrete graphs which respectively correspond to the marked and unmarked parts of $\alpha$.
	We will write $k_{\alpha}$ for the sum $n_1+ \cdots + n_k$ and denote by $\Fr\:\Graph \to \Cat$ the free category functor.
	We define the epimorphism $\phi(\alpha)\:\Fr[n] \twoheadrightarrow \Fr[n-k_{\alpha}]$ as the canonical morphism coming from the pushout
	\begin{center}
	\begin{tikzcd}
		\Fr\marking{\alpha}
			\ar[r,"\Fr i_{\alpha}", hook]
			\ar[d, twoheadrightarrow]
			\ar[dr,"\ulcorner" description, very near end, phantom] &
		\Fr[n]
			\ar[d,"\phi(\alpha)", twoheadrightarrow] \\
		\pi_0(\Fr\marking{\alpha})
			\ar[r] &
		\Fr[n-k_{\alpha}]
	\end{tikzcd}
	\end{center}
	where $\pi_0$ returns the set of connected components viewed as a discrete category.
	Let $g\:\alpha \to \alpha'$ be a map in $\Ar{\Alt}{\rfree}$ and construct the following diagram in~$\Cat$:
	\begin{equation}\label{diag:phimor}
	\begin{tikzcd}[row sep=0.7cm, column sep=0.8cm]
		& \Fr\marking{\alpha}
			\ar[dl, hook']
			\ar[rr, "\Fr\marking{g}" description]
			\ar[dd, twoheadrightarrow]
			\ar[dddl,"\urcorner" description, very near end,phantom]
		&& \Fr\marking{\alpha'}
			\ar[dl, hook']
			\ar[dd, twoheadrightarrow]
			\ar[dddl,"\urcorner" description, very near end, phantom] \\
		{\Fr[n]}
			\ar[rr, "\Fr\carrier{g}" description, near start, crossing over]
			\ar[dd, "\phi(\alpha)" description, twoheadrightarrow]
		&& {\Fr[n']} \\
		& \pi_0(\Fr\marking{\alpha})
			\ar[dl, hook']
			\ar[rr, "\pi_0(\Fr\marking{g})" description, near start]
		&& \pi_0(\Fr\marking{\alpha'})
			\ar[dl, hook'] \\
		{\Fr[n-k_{\alpha}]}
			\ar[rr, "\fbot{\phi(g)}" description, dotted]
		&& {\Fr[n'-k_{\alpha'}]}
			\ar[from=uu, "\phi(\alpha')" description, near start, crossing over]
	\end{tikzcd}
	\end{equation}
	The front square exhibits a map $\phi(g)\:\phi(\alpha) \to \phi(\alpha')$ in $\fDel_0$.
	In fact, $\phi(g)$ is distance preserving since $\carrier{g}$ is distance preserving monomorphism by \cref{rmk:GraphMorphDistPres} and the maps $\fbot{\phi(g)}$ and $\Fr\carrier{g}$ are respectively in $\Delta$ and $\Delta_{+}$.

\end{constr}

This construction defines a functor $\phi\:\Ar{\Alt}{\rfree} \to \fDel_0$ since the morphism part is defined by a universal property.
Now, we construct a functor ${\psi\:\fDel_0 \to \Ar{\Alt}{\rfree}}$ which we will show to be the inverse of $\phi$.

\begin{constr}\label{constr:psi}
	Let $\eta\:[m] \twoheadrightarrow [n]$ be an object in $\fDel_0$ and consider $\forgetful{\eta}$, where the forgetful functor $\forgetful{(\blank)}\: \Cat \to \Graph$ is the right adjoint to $\Fr$.
	Recall that $[m]$ denotes both the ordinal of length $m$ and the linearly ordered graph of length $m$, with the context clarifying the intended meaning.
	Since the ordinal $[m]$ is the free category on the linear graph of length $m$, applying the forgetful functor $\forgetful{(\blank)}$ to an ordinal amounts to applying the free category monad $\frc$ to the corresponding linear graph.
	So, using the unit $\nu\:\id \to \frc$, we can consider the composite $\forgetful{\eta}\nu_{[m]}\:[m] \to \frc[m] \to \frc[n]$ and construct the pullback in $\Graph$
	\begin{equation*}
	\begin{tikzcd}
		\marking{\psi(\eta)}
			\ar[r, hook]
			\ar[dd]
			\ar[ddr, phantom, very near start, "{ \lrcorner }"] &
		{[m]}
			\ar[d, "\nu_{[m]}"] \\
		&
		{\frc[m]}
			\ar[d, "\forgetful{\eta}"] \\
		{\frc\{n\}}
			\ar[r, "\frc i_n"', hook] &
		{\frc[n]}
	\end{tikzcd}
	\end{equation*}
	where $\{n\}$ is the discrete graph with $n+1$ vertices and $i_n$ is the inclusion of the vertices.
	The graph $\marking{\psi(\eta)}$ corresponds to the wide subgraph of edges in the graph $[m]$ that are sent to identities in $[n]$ by $\eta$.
	From a morphism $g\:\eta \to \eta'$ in $\fDel_0$, we construct a morphism in $\Ar{\Alt}{\rfree}$ in a similar fashion as in \eqref{diag:phimor} by assembling the following diagram in $\Graph$ whose side squares are pullbacks:
	\begin{equation}\label{diag:psimor}
	\begin{tikzcd}[row sep=0.5cm, column sep=0.3cm]
		&& {\frc\{n\}}
			\ar[rrr]
			\ar[ddd, hook']
		&&& {\frc\{n'\}}
			\ar[ddd, hook]\\
		&&&&&\\
		\marking{\psi(\eta)}
			\ar[uurr]
			\ar[ddd, "\psi(\eta)" description, hook']
			\ar[rrr, "\marking{\psi(g)}" description, dotted, crossing over]
		&&& \marking{\psi(\eta')}
			\ar[rruu]
			\ar[ddd, "\psi(\eta')" description, hook, crossing over]
		&&\\
		&& {\frc[n]}
			\ar[rrr, "\forgetful{\fbot{g}}" description]
		&&& {\frc[n']}\\
		& {\frc[m]}
			\ar[ur, "\forgetful{\eta}"]
			\ar[rrr, "\forgetful{\ftop{g}}" description]
		&&& {\frc[m']}
			\ar[ur, "\forgetful{\eta'}"']
		&\\
		{[m]}
			\ar[ur, "\nu_{[m]}"]
			\ar[rrr, "\carrier{\psi(g)}" description]
		&&& {[m']}
			\ar[ur, "\nu_{[m']}"']
			\ar[from=uuu, "\psi(\eta')" description, hook, crossing over]
		&&
	\end{tikzcd}
	\end{equation}
	Since $g$ is in $\fDel_0$, the morphism $\ftop{g}\:[m] \hookrightarrow [m']$ is inert in $\Delta$ and is thus uniquely determined by a graph morphism denoted by $\carrier{\psi(g)}\:[m] \to [m']$.
	For more details about the connections between inert morphisms (defined as free morphisms) and the category of arities in general, we refer to \cite[Section 3]{BMW2012}.
	In particular, we have $\forgetful{\ftop{g}}=\frc\carrier{\psi(g)}$ and, by naturality of the unit, $\nu_{[m']}\carrier{\psi(g)}=\forgetful{\ftop{g}}\nu_{[m]}$.
	Using $\frc\carrier{\psi(g)}$ and the map $\{n\} \to \{n'\}$ induced by $\fbot{g}$ on the vertices, we can construct the dotted map $\marking{\psi(g)}$ by the universal property of pullbacks.
	The front square exhibits a map $\psi(g)\: \psi(\eta) \to \psi(\eta')$ in $\RelGraph$.	
\end{constr}

\begin{thm}\label{thm:arfDel}
	The category of arities $\Ar{\Alt}{\rfree}$ of the monad $\rfree$ and $\fDel_0$ are isomorphic.
\end{thm}
\begin{proof}
	We show that the functors $\phi\:\Ar{\Alt}{\rfree} \to \fDel_0$ and $\psi\:\fDel_0 \to \Ar{\Alt}{\rfree}$ are inverse to each other.
	Let us denote by $\alpha$ an object $\marking{\alpha} \hookrightarrow [n]$ in $\Ar{\Alt}{\rfree}$. By the \cref{constr:phi,constr:psi} of $\phi$ and $\psi$, the canonical arity $\psi\phi(\alpha)$ is the embedding $\marking{\psi\phi(\alpha)} \hookrightarrow [n]$. The wide subgraph $\marking{\psi\phi(\alpha)}$ represents the edges in $[n]$ which are sent to identities by $\phi(\alpha)$.
	Since the latter is a quotient map with respect to $\marking{\alpha}$, we get that $\marking{\psi\phi(\alpha)}$ is exactly $\marking{\alpha}$, thus $\psi\phi(\alpha)=\alpha$.
	Let $g\:\alpha \to \alpha'$ be a morphism in $\Ar{\Alt}{\rfree}$. Since $\carrier{g}$ is distance preserving by \cref{rmk:GraphMorphDistPres}, the map $\ftop{\phi(g)}$ is inert in $\Delta$.
	Therefore, the morphism $\forgetful{\ftop{\phi(g)}}$ is uniquely determined by $\carrier{g}$, i.e.\ $\forgetful{\ftop{\phi(g)}}=\forgetful{(\Fr\carrier{g})}=\frc \carrier{g}$.
	In particular, by naturality of $\nu$ we have
	\begin{equation*}
	\begin{tikzcd}
		{[n]}
			\ar[r, "\carrier{g}"]
			\ar[d, "\nu_{[n]}"']
		& {[n']}
			\ar[d, "\nu_{[n']}"] \\
		{\frc[n]}
			\ar[r, "\frc\carrier{g}"']
		& {\frc[n']}
	\end{tikzcd}
	\end{equation*}
	so $\carrier{\psi\phi(g)}=\carrier{g}$, and $g$ preserves the marking as $\marking{\psi\phi(g)}=\marking{g}$ by uniqueness of the dotted arrow in \eqref{diag:psimor}.
	So, we obtain the required equality, $\psi\phi=\id_{\Alt_{\rfree}}$.

	Let $\eta\:[m] \twoheadrightarrow [n]$ be an object in $\fDel_0$.
	By \cref{constr:phi}, we have that~$\phi\psi(\eta)$ is the quotient map associated with $[m]$ and with respect to $\marking{\psi(\eta)}$.
	The wide subgraph $\marking{\psi(\eta)}$ of $[m]$ is identified with the edges that are mapped to identities by $\eta$, thus collapsing them in the ordinal $[m]$ produces the category $[n]$ and the quotient map $\phi\psi(\eta)$ is $\eta$.
	For a map $g\:\eta \to \eta'$ in $\fDel_0$, we obtain~$\ftop{\phi\psi(g)}=\Fr\carrier{\psi(g)}$ by construction of $\phi$, which is also equal to $\ftop{g}$ since $g$ is in $\fDel_0$ (the morphism $\ftop{g}$ is inert in $\Delta$). Finally, the morphism $\fbot{\phi\psi(g)}$ is $\fbot{g}$ by the uniqueness of the dotted arrow in \eqref{diag:phimor}.
	Thus, $\phi\psi=\id_{\fDel_0}$ and we have~$\Ar{\Alt}{\rfree}\cong\fDel_0$.
\end{proof}

\subsection{Main results from monads with arities}\label{sec:nervethm}
We have seen that $\fDel_0$ is isomorphic to the category of arities of the monad~$\rfree$, and from now on we will say that the monad $\rfree$ has arities $\fDel_0$.
We will now show that the theory $\Theta_{\rfree}$  associated with $\rfree$ (see \cref{sec:mwa}) can be identified with $\fDel$.
Consider the composition of the functors
\begin{equation*}
	\fDel_0 \xhookrightarrow{\emb{i}_0} \RelGraph \xrightarrow{\rFr} \RelSemiCat.
\end{equation*}
The category $\Theta_{\rfree}$ is defined by factoring this composition into a bijective-on-objects functor followed by a fully faithful functor.
The inclusion $\emb{j}\:\fDel_0 \hookrightarrow \fDel$ is the identity-on-objects as $\fDel_0$ is a wide subcategory of $\fDel$.

\begin{lem}
	The canonical functor $\emb{i}\:\fDel \to \RelSemiCat$ is fully faithful.
\end{lem}
\begin{proof}
	The functor $\emb{i}$ maps an object $\eta\:[m] \twoheadrightarrow [n]$ to the linear semicategory $[m]$ equipped with the wide subsemicategory of non-trivial edges that are sent to identities by $\eta$ as marked morphisms.
	Thus, the induced map on morphisms $\emb{i}\: \fDel(\eta,\eta') \to \RelSemiCat(\emb{i}(\eta),\emb{i}(\eta'))$ sends $f$ to the top arrow $\ftop{f}$.
	This is well-defined, as the commutativity of the square $f$ ensures that $\ftop{f}$ preserves the marking of the edges and, since $\ftop{f}$ is a monomorphism, it restricts to linear semicategories.
	The fact that $\fDel(\eta,\eta') \to \RelSemiCat(\emb{i}(\eta),\emb{i}(\eta'))$ is an isomorphism follows from \cref{rmk:bodet}.
\end{proof}

In particular, the composition $\fDel_0 \xrightarrow{\emb{j}} \fDel \xhookrightarrow{\emb{i}} \RelSemiCat$ factors $\rFr \emb{i}_0$ as a bijective-on-objects followed by a fully faithful functor.
Therefore, we have the following result:

\begin{thm}\label{thm:fDeldesc}
	The category $\Theta_{\rfree}$ is isomorphic to $\fDel$.\qed
\end{thm}

Since $\Theta_{\rfree}$ corresponds to the category of free relative semicategories over the category of relative linear graphs, the theory associated with $\rfree$ is the category of non-empty finite relative semiordinals.
This gives a rigorous proof of the intuition conveyed in \cite{Kock2006,Pao2025} and \cref{sec:fDeldef} that $\fDel$ can either be described in terms of relative semiordinals or epimorphisms of $\Delta$.

The main result of this section, the nerve theorem for $\fDel$, is a direct consequence of \cref{thm:nerve,thm:fDeldesc}.

\begin{thm}\label{thm:fnerve}
	The nerve functor $\fNerve\:\RelSemiCat \to \Psh{\fDel}$ is fully faithful.
	The essential image is spanned by the presheaves whose restriction along $\emb{j}$ belongs to the essential image of $\fNerve_0\:\RelGraph \to \Psh{\fDel_0}$.
\end{thm}
\begin{proof}
	The monad $\rfree$ has arities $\fDel_0$ by \cref{thm:arfDel} and the category of free $\rfree$-algebras over $\fDel_0$ is isomorphic to $\fDel$ by \cref{thm:fDeldesc}.
	Thus, the result follows from \cref{thm:nerve}.
\end{proof}
For comparison, we also provide a direct proof of the first part of \cref{thm:fnerve}.
\begin{proof}[Alternative proof of the first part of \cref{thm:fnerve}]
	Let $\C$ and $\D$ be two relative semicategories; we need to show that the hom-set $\RelSemiCat(\C,\D)$ is canonically equivalent to the set of natural transformations $\Psh{\fDel}(\fNerve \C,\fNerve \D)$.
	Let $\phi$ be such a natural transformation.
	Notice that $\phi_{[0]^{\flat}}\:(\fNerve \C)_{[0]^{\flat}} \to (\fNerve \D)_{[0]^{\flat}}$ is a map between the set of objects of $\C$ and $\D$, and $\phi_{[1]^{\flat}}\: (\fNerve \C)_{[1]^{\flat}} \to (\fNerve \D)_{[1]^{\flat}}$ is a map from the set of all morphisms in $\C$ to the set of all morphisms in~$\D$.
	So naturality, applied on the face maps $d_0$ and $d_1$, ensures that the pair of $\phi_{[0]^{\flat}}$ and $\phi_{[1]^{\flat}}$ is a graph morphism. Similarly, $\phi_{[2]^{\flat}}$ means that the graph morphism respects composition.
	Finally, $\phi_{[1]^{\sharp}}$ together with the naturality square that we get from the unique morphism in $\fDel^{\op}([1]^{\sharp},[1]^{\flat})$, ensures that our graph morphism sends marked morphisms in $\C$ to marked morphisms in $\D$.
	This data exactly determines a functor $f\:\C \to \D$ between the given relative semicategories, and it is clear that $\fNerve(f)$ gives this data.

	What is left to check is that $\phi_{\eta}$, for any object $\eta \in \fDel$, as well as the corresponding naturality conditions, does not add any non-trivial information to the above (i.e., the space of extensions of a functor to a natural transformation is contractible); but this follows from the 2-coskeletality of $\D$.
	More precisely, as every relative semiordinal $\emb{i}(\eta)$ is free on its relative linear graph, the value of the nerve at $\eta$ is simply the corresponding ``Segal limit.''
\end{proof}

\begin{thm}\label{thm:fDelhomo}
	The theory associated with the monad $\rfree$, that is $\fDel$, is homogeneous.
\end{thm}
\begin{proof}
	By \cref{prop:rfreeStrgCart}, the monad $\rfree$ is strongly cartesian, and by \cref{thm:fDeldesc} the theory associated to $\rfree$ is $\fDel$. Thus, the result follows from \cref{thm:homog}.
\end{proof}

\cref{thm:fDelhomo} and the definition of a homogeneous theory in \cref{sec:mwa} yield an orthogonal factorisation system on $\fDel$ given by generic and free morphisms.
By \cref{rmk:GraphMorphDistPres} and unfolding the construction of \cref{thm:arfDel}, the set of free morphisms can be seen to be the inert part of $\fDel$, introduced in \cref{sec:activeinert}.
The set of generic morphisms can be proven to be exactly the active part of $\fDel$ by a direct computation.
Although not difficult, the argument is rather technical, and thus we prefer to give a direct construction of the orthogonal factorisation system.
By uniqueness of the left class in an orthogonal factorisation system, it follows that the generic morphisms are precisely the active ones.

\begin{prop}\label{cor:actinfDel}
  The active and inert maps form an orthogonal factorisation system $(\fDel_a, \fDel_0)$ on $\fDel$.
\end{prop}
\begin{proof}
We prove that any map in $\fDel$ factors functorially through the classes $\fDel_a$ and $\fDel_0$, which by \cite[Lemma 1.8]{Riehl2008} is equivalent to the usual notion of an orthogonal factorisation system.
Clearly, $\fDel_a$ and $\fDel_0$ contain all isomorphisms and are closed under composition since the same holds for the active-inert factorisation system in~$\Delta$.
Let us show that every map in $\fDel$ factors as a map in $\fDel_a$ followed by a map in $\fDel_0$.
Consider a map $f\: \eta_0 \to \eta_1$ in $\fDel$, with $\eta_0\:[m_0] \twoheadrightarrow [n_0]$ and $\eta_1\:[m_1] \twoheadrightarrow [n_1]$.
First, we factor the monomorphism $\ftop{f}$ in~$\Delta$ into an active  $\ftop{a}\:[m_0] \hookrightarrow [m]$ followed by an inert map $\ftop{i}\: [m] \hookrightarrow [m_1]$.
Since $\ftop{i}\,\ftop{a}=\ftop{f}$ and $\ftop{i}$ are monomorphisms (any inert map in $\Delta$ is a monomorphism), so is $\ftop{a}$.
Now, we construct the following diagram:
\begin{center}
\begin{tikzcd}
	{[m_0]}
		\ar[r, "\ftop{a}" description, hook]
		\ar[d, "\eta_0"', twoheadrightarrow]
		\ar[rr, "\ftop{f}" description, bend left] &
	{[m]}
		\ar[r, "\ftop{i}" description, hook]
		\ar[d, "\eta", twoheadrightarrow] &
	{[m_1]}
		\ar[d, "\eta_1", twoheadrightarrow] \\
	{[n_0]}
		\ar[r, "\fbot{a}" description]
		\ar[rr, "\fbot{f}" description, bend right] &
	{[n]}
		\ar[r, "\fbot{i}" description, dotted]
		\ar[ul, phantom, very near start, "{\ulcorner}"] &
	{[n_1]}
\end{tikzcd}
\end{center}
The pushout exists by \cite[Corollary 3.3]{CFPS2023}.
It remains to show the functoriality of the factorisation, that is, given two active-inert factorisations and maps as in the following solid diagram
\begin{equation*}
\begin{tikzcd}[row sep=0.7cm, column sep=0.9cm]
	& {[m_0]}
		\ar[dl, "\ftop{f_0}" description, hook']
		\ar[rr, "\ftop{a}" description, hook]
		\ar[dd, "\eta_0" description, near start, twoheadrightarrow]
		\ar[ddrr, "\ulcorner" description, very near end, phantom]
	&& {[m_1]}
		\ar[dl, "\ftop{f_1}" description, dotted]
		\ar[dd, "\eta_1" description, near start, twoheadrightarrow]
		\ar[rr, "\ftop{i}" description, hook]
	&& {[m_2]}
		\ar[dl, "\ftop{f_2}" description, hook']
		\ar[dd, "\eta_2" description, twoheadrightarrow]\\
	{[m'_0]}
		\ar[rr, "\ftop{a'}" description, near end, crossing over, hook]
		\ar[dd, "\eta'_0" description, twoheadrightarrow]
		\ar[ddrr, "\ulcorner" description, very near end, phantom]
	&& {[m'_1]}
		\ar[rr, "\ftop{i'}" description, near end, crossing over, hook]
	&& {[m'_2]}\\
	& {[n_0]}
		\ar[dl, "\fbot{f_0}" description]
		\ar[rr, "\fbot{a}" description, near start]
	&& {[n_1]}
		\ar[dl, "\fbot{f_1}" description, dotted]
		\ar[rr, "\fbot{i}" description, near start]
	&& {[n_2]}
		\ar[dl, "\fbot{f_2}" description]\\
	{[n'_0]}
		\ar[rr, "\fbot{a'}" description, crossing over]
	&& {[n'_1]}
		\ar[from=uu, "\eta'_1" description, near start, twoheadrightarrow, crossing over]
		\ar[rr, "\fbot{i'}" description, crossing over]
	&& {[n'_2]}
		\ar[from=uu, "\eta'_2" description, near start, twoheadrightarrow, crossing over]
\end{tikzcd}
\end{equation*}
there is a unique $f_1\:\eta_1 \to \eta_1'$ making the squares commute.
Since $\ftop{a}$ and $\ftop{a'}$ are active, and $\ftop{i}$ and $\ftop{i'}$ are inert, functoriality of the active-inert factorisation system in $\Delta$ applied to the top diagram yields a unique map $\ftop{f_1}\:[m_1] \to [m'_1]$ making the top squares commute.
In particular, we have $\ftop{f_2}\,\ftop{i}=\ftop{i'}\,\ftop{f_1}$, so since $\ftop{f_2}\,\ftop{i}$ is a monomorphism by composition, $\ftop{f_1}$ is also monic.
Finally, $\ftop{f_1}$ and the universal property of the left-back pushout square yield a unique morphism $\fbot{f_1}\:[n_1] \to [n'_1]$ in $\Delta$ making everything commute.
\end{proof}


%% file: biblio.bib
@inproceedings{ACK2016,
  author    = {Altenkirch, Thorsten and Capriotti, Paolo and Kraus, Nicolai},
  title     = {Extending Homotopy Type Theory with Strict Equality},
  booktitle = {25th EACSL Annual Conference on Computer Science Logic (CSL 2016)},
  pages     = {21:1--21:17},
  series    = {Leibniz International Proceedings in Informatics (LIPIcs)},
  editor    = {Jean-Marc Talbot and Laurent Regnier},
  year      = {2016},
  volume    = {62},
  publisher = {Schloss Dagstuhl -- Leibniz-Zentrum f{\"u}r Informatik},
  doi       = {10.4230/LIPIcs.CSL.2016.21}
}

@article{ACKS2023,
  title        = {Two-Level Type Theory and Applications},
  author       = {Annenkov, Danil and Capriotti, Paolo and Kraus, Nicolai and Sattler, Christian},
  year         = {2023},
  journaltitle = {Mathematical Structures in Computer Science},
  volume       = {33},
  number       = {8},
  pages        = {688--743},
  doi          = {10.1017/S0960129523000130}
}

@article{ACU2015,
  title   = {Monads need not be endofunctors},
  volume  = {11},
  number  = {1},
  doi     = {10.2168/lmcs-11(1:3)2015},
  journal = {Logical Methods in Computer Science},
  author  = {Altenkirch, Thorsten and Chapman, James and Uustalu, Tarmo},
  year    = {2015}
}

@article{Ber2002,
author = {Clemens Berger},
title = {A Cellular Nerve for Higher Categories},
journal = {Advances in Mathematics},
volume = {169},
number = {1},
pages = {118-175},
year = {2002},
doi = {10.1006/aima.2001.2056},
}

@article{Ber2023,
  title   = {Moment categories and operads},
  author  = {Berger, Clemens},
  journal = {Theory and Applications of Categories},
  volume  = {38},
  number  = {39},
  pages   = {1485--1537},
  year    = {2022}
}

@incollection{Bergner2011,
  title     = {Models for $(\infty, n)$-categories and the cobordism hypothesis},
  author    = {Bergner, Julia E},
  booktitle = {Mathematical foundations of quantum field theory and perturbative string theory},
  series    = {Proceedings of Symposia in Pure Mathematics},
  editor    = {Hisham Sati and Urs Schreiber},
  volume    = {83},
  pages     = {17--30},
  year      = {2011},
  doi       = {10.1090/pspum/083/2742424},
  publisher = {American Mathematical Society}
}

@article{BK2012,
  title        = {Relative Categories: {{Another}} Model for the Homotopy Theory of Homotopy Theories},
  shorttitle   = {Relative Categories},
  author       = {Barwick, C. and Kan, D. M.},
  year         = {2012},
  journaltitle = {Indagationes Mathematicae},
  volume       = {23},
  number       = {1},
  doi          = {10.1016/j.indag.2011.10.002}
}

@article{BMW2012,
  title        = {Monads with Arities and Their Associated Theories},
  author       = {Berger, Clemens and Melliès, Paul-André and Weber, Mark},
  year         = {2012},
  journaltitle = {Journal of Pure and Applied Algebra},
  volume       = {216},
  number       = {8},
  doi          = {10.1016/j.jpaa.2012.02.039}
}

@book{Borceux1994,
  title     = {Handbook of Categorical Algebra},
  volume    = {2},
  subtitle  = {Categories and Structures},
  series    = {Encyclopedia of Mathematics and its Applications},
  number    = {51},
  doi       = {10.1017/cbo9780511525865},
  publisher = {Cambridge University Press},
  author    = {Borceux, Francis},
  year      = 1994
}

@article{CFPS2023,
  title        = {Weak Cartesian Properties of Simplicial Sets},
  author       = {Constantin, Carmen and Fritz, Tobias and Perrone, Paolo and Shapiro, Brandon T.},
  year         = {2023},
  journaltitle = {Journal of Homotopy and Related Structures},
  volume       = {18},
  number       = {4},
  doi          = {10.1007/s40062-023-00334-1}
}

@article{CH2021,
  title        = {Homotopy-coherent algebra via {Segal} conditions},
  author       = {Chu, Hongyi and Haugseng, Rune},
  year         = {2021},
  journaltitle = {Advances in Mathematics},
  volume       = {385},
  doi          = {10.1016/j.aim.2021.107733}
}

@misc{CSZ2026,
      title={The free bifibration on a functor}, 
      author={Bryce Clarke and Gabriel Scherer and Noam Zeilberger},
      year={2026},
      doi={10.48550/arXiv.2511.07314} 
}

@book{Elephant2,
  title     = {Sketches of an Elephant: A Topos Theory Compendium},
  author    = {Johnstone, Peter T},
  volume    = {2},
  year      = {2002},
  publisher = {Oxford University Press},
  doi       = {10.1093/oso/9780198515982.001.0001}
}

@article{GKT2018a,
  title        = {Decomposition spaces, incidence algebras and {{Möbius}} inversion {{I}}: {{Basic}} theory},
  author       = {G\'alvez-Carrillo, Imma and Kock, Joachim and Tonks, Andrew},
  year         = {2018},
  journaltitle = {Advances in Mathematics},
  volume       = {331},
  doi          = {10.1016/j.aim.2018.03.016},
  pages        = {952--1015}
}

@misc{GP2022,
  title  = {The geometric cobordism hypothesis},
  author = {Daniel Grady and Dmitri Pavlov},
  year   = {2022},
  doi    = {10.48550/arXiv.2111.01095}
}

@article{Harpaz_2015,
  title     = {Quasi-unital $\infty$-categories},
  volume    = {15},
  doi       = {10.2140/agt.2015.15.2303},
  number    = {4},
  journal   = {Algebraic \& Geometric Topology},
  publisher = {Mathematical Sciences Publishers},
  author    = {Harpaz, Yonatan},
  year      = {2015},
  pages     = {2303--2381}
}

@article{Haugseng_2021,
  title   = {Segal spaces, spans, and semicategories},
  author  = {Haugseng, Rune},
  journal = {Proceedings of the American Mathematical Society},
  volume  = {149},
  number  = {3},
  pages   = {961--975},
  year    = {2021},
  doi     = {10.1090/proc/15197}
}

@article{Henry2018,
  author = {{Henry}, Simon},
  title  = {{Regular polygraphs and the Simpson conjecture}},
  year   = 2018,
  doi    = {10.48550/arXiv.1807.02627}
}

@article{HK2022,
  title   = {$\infty$-operads as symmetric monoidal $\infty$-categories},
  author  = {Haugseng, Rune and Kock, Joachim},
  year    = 2024,
  pages   = {111--137},
  volume  = 68,
  number  = 1,
  journal = {Publicacions Matem\`atiques},
  doi     = {10.5565/PUBLMAT6812406}
}

@book{hottbook,
  author    = {The {Univalent Foundations Program}},
  title     = {Homotopy Type Theory: Univalent Foundations of Mathematics},
  publisher = {\url{https://homotopytypetheory.org/book}},
  address   = {Institute for Advanced Study},
  year      = 2013
}

@incollection{JK2007,
  author    = {Joyal, Andr{\'e} and Kock, Joachim},
  title     = {Weak units and homotopy 3-types},
  booktitle = {Categories in algebra, geometry and mathematical physics},
  series    = {Contemp. Math.},
  volume    = {431},
  pages     = {257--276},
  publisher = {Amer. Math. Soc., Providence, RI},
  year      = {2007},
  doi       = {10.1090/conm/431/08277}
}

@inproceedings{Joyal1986,
  title     = {Foncteurs analytiques et esp\`eces de structures},
  booktitle = {Combinatoire \'enum\'erative},
  series    = {Lecture Notes in Mathematics},
  volume    = {1234},
  author    = {Joyal, Andr\'e},
  date      = {1986},
  publisher = {Springer},
  doi       = {10.1007/BFb0072514}
}

@unpublished{Joyal2008,
  author = {Joyal, Andr\'{e}},
  note   = {Lecture notes for an Advanced Course on Simplicial Methods in Higher Categories at the Centre de Recerca Matem\`atica (CRM)},
  title  = {The Theory of Quasi-Categories and its Applications},
  year   = {2008},
  url    = {http://mat.uab.cat/~kock/crm/hocat/advanced-course/Quadern45-2.pdf}
}

@article{Kock2006,
  title   = {Weak identity arrows in higher categories},
  author  = {Kock, Joachim},
  year    = {2006},
  journal = {International Mathematics Research Papers},
  volume  = {2006},
  pages   = {1--54},
  doi     = {10.1155/IMRP/2006/69163}
}

@misc{KockCom,
  author       = {Joachim Kock},
  date         = {2010-08-01},
  howpublished = {Comment on an n-Category Caf\'e's blog post},
  title        = {On Ternary Factorization Systems},
  url          = {https://golem.ph.utexas.edu/category/2010/07/ternary_factorization_systems.html#c034162}
}

@online{KS2017,
  title  = {Space-Valued Diagrams, Type-Theoretically (Extended Abstract)},
  author = {Kraus, Nicolai and Sattler, Christian},
  date   = {2017},
  doi    = {10.48550/arXiv.1704.04543}
}

@book{Leinster2004,
  title     = {Higher Operads, Higher Categories},
  author    = {Leinster, Tom},
  date      = {2004},
  series    = {London Mathematical Society Lecture Note Series},
  volume    = {298},
  publisher = {Cambridge University Press},
  doi       = {10.1017/CBO9780511525896}
}

@book{Loregian2021,
  title     = {Co(end) Calculus},
  author    = {Loregian, Fosco},
  date      = {2021},
  series    = {London Mathematical Society Lecture Note Series},
  volume    = {468},
  publisher = {Cambridge University Press},
  doi       = {10.1017/9781108778657}
}

@article{Lurie2008,
  title   = {On the classification of topological field theories},
  author  = {Lurie, Jacob},
  journal = {Current Developments in Mathematics},
  volume  = {2008},
  number  = {1},
  pages   = {129--280},
  year    = {2008},
  doi     = {10.4310/CDM.2008.v2008.n1.a3}
}

@unpublished{LurieHA,
  title  = {Higher Algebra},
  author = {Lurie, Jacob},
  date   = {2017},
  url    = {https://www.math.ias.edu/~lurie/papers/HA.pdf}
}

@book{LurieHTT,
  title         = {Higher Topos Theory},
  author        = {Lurie, Jacob},
  year          = {2009},
  publisher     = {Princeton University Press},
  series        = {Annals of Mathematics Studies},
  volume        = {170},
  doi           = {10.1515/9781400830558},
  eprint        = {math/0608040},
  archiveprefix = {arXiv}
}

@book{MartinLof1984,
  title     = {Intuitionistic Type Theory},
  author    = {Martin-L\"of, Per},
  year      = 1984,
  series    = {Studies in Proof Theory},
  volume    = 1,
  note      = {Notes by Giovanni Sambin of a series of lectures given in
               Padua, June 1980},
  publisher = {Bibliopolis},
  url       = {https://raw.githubusercontent.com/michaelt/martin-lof/master/pdfs/Bibliopolis-Book-retypeset-1984.pdf}
}

@book{NPS90,
  title     = {Programming in Martin-L\"of's type theory},
  subtitle  = {An introduction},
  author    = {Nordstr\"om, Bengt and Petersson, Kent and Smith, Jan M.},
  year      = 1990,
  publisher = {Clarendon Press, Oxford University Press},
  series    = {International Series of Monographs on Computer Science},
  number    = 7,
  url       = {https://www.cse.chalmers.se/research/group/logic/book/book.pdf}
}

@article{Pao2025,
title={Weakly globular double categories and weak units},
volume={9},
DOI={10.21136/hs.2025.06},
number={1},
journal={Higher Structures},
author={Paoli, Simona},
year={2025},
pages={269--328}
}

@book{Paoli2019,
  title     = {Simplicial Methods for Higher Categories},
  subtitle  = {Segal-type Models of Weak $n$-Categories},
  author    = {Paoli, Simona},
  series    = {Algebra and Applications},
  volume    = {26},
  year      = {2019},
  publisher = {Springer},
  doi       = {10.1007/978-3-030-05674-2}
}

@book{Reflections,
  title     = {Reflections on the Foundations of Mathematics},
  subtitle  = {Univalent Foundations, Set Theory and General Thoughts},
  editor    = {Centrone, Stefania and Kant, Deborah and Sarikaya, Deniz},
  publisher = {Springer},
  series    = {Synthese Library},
  volume    = {407},
  edition   = 1,
  year      = 2019,
  doi       = {10.1007/978-3-030-15655-8}
}

@article{Rezk2001,
  title   = {A model for the homotopy theory of homotopy theory},
  author  = {Rezk, Charles},
  journal = {Transactions of the American Mathematical Society},
  volume  = {353},
  number  = {3},
  pages   = {973--1007},
  year    = {2001},
  doi     = {10.1090/S0002-9947-00-02653-2}
}

@online{Riehl2008,
  title  = {Factorization systems},
  author = {Riehl, Emily},
  year   = {2008},
  url    = {https://emilyriehl.github.io/files/factorization.pdf}
}

@unpublished{Sattler2017,
  title  = {Kock's Fat $\Delta$ Is a Direct Replacement of $\Delta$.},
  author = {Sattler, Christian},
  date   = {2017},
  url    = {https://www.cse.chalmers.se/~sattler/docs/fat-delta.pdf}
}

@thesis{ShapiroPhD,
  title       = {Shape Independent Category Theory},
  author      = {Shapiro, Brandon T.},
  type        = {phdthesis},
  date        = {2022},
  institution = {Cornell University},
  doi         = {10.7298/fcxt-vw82}
}

@article{Shulman2015,
  title        = {Univalence for Inverse Diagrams and Homotopy Canonicity},
  author       = {Shulman, Michael},
  year         = {2015},
  journaltitle = {Mathematical Structures in Computer Science},
  volume       = {25},
  number       = {5},
  doi          = {10.1017/S0960129514000565}
}

@inbook{Shulman2021,
  title     = {Homotopy Type Theory: The Logic of Space},
  booktitle = {New Spaces in Mathematics: Formal and Conceptual Reflections},
  editor    = {Mathieu Anel and Gabriel Catren},
  publisher = {Cambridge University Press},
  author    = {Shulman, Michael},
  year      = {2021},
  pages     = {322--404},
  doi       = {10.1017/9781108854429.009}
}

@online{Sim1998,
  title  = {Homotopy types of strict 3-groupoids},
  author = {Simpson, Carlos},
  year   = {1998},
  doi    = {10.48550/arXiv.math/9810059}
}

@article{Tamsamani1999,
  title    = {Sur des notions de $n$-categorie et $n$-groupoide non strictes via des ensembles multi-simpliciaux},
  subtitle = {(On the notions of a nonstrict $n$-category and $n$-groupoid via multisimplicial sets)},
  author   = {Tamsamani, Zouhair},
  journal  = {K-theory},
  volume   = {16},
  number   = {1},
  pages    = {51--99},
  year     = {1999},
  doi      = {10.1023/a:1007747915317}
}

@book{TTT1985,
  author    = {Michael Barr and Charles Wells},
  title     = {Toposes, Triples and Theories},
  year      = 1985,
  publisher = {Springer-Verlag},
  series    = {Grundlehren der mathematischen Wissenschaften},
  volume    = 278,
  addendum  = {Republished in: \emph{Reprints in Theory and Applications of
               Categories, No.~12 (2005) pp.~1--287}}
}

@article{Weber2004,
  title        = {Generic Morphisms, Parametric Representations and Weakly Cartesian Monads.},
  author       = {Weber, Mark},
  date         = {2004},
  journaltitle = {Theory and Applications of Categories},
  volume       = {13}
}

@article{Weber2007,
  title        = {Familial 2-Functors and Parametric Right Adjoints.},
  author       = {Weber, Mark},
  date         = {2007},
  journaltitle = {Theory and Applications of Categories},
  volume       = {18}
}
